\documentclass[twoside]{irmaems}
\usepackage{amssymb} 
\usepackage{amsmath} 
\usepackage{latexsym}
\usepackage{graphicx} 

\renewcommand\theequation{\arabic{section}.\arabic{equation}}

\theoremstyle{definition} 

 \newtheorem{definition}{Definition}[section]
 \newtheorem{remark}[definition]{Remark}

\theoremstyle{plain}      

 \newtheorem{proposition}[definition]{Proposition}
 \newtheorem{theorem}[definition]{Theorem}
 \newtheorem{corollary}[definition]{Corollary}
 \newtheorem{lemma}[definition]{Lemma}

\def\Z{\mathbb{Z}}

\def\rb{\mathbb{R}}

\def\zb{\mathbb{Z}}

\def\mc{\mathcal{M}}

\def\rc{\mathcal{R}}

\begin{document}

\title{
On the spectrum of the Thue-Morse quasicrystal
and the rarefaction phenomenon}

\author{Jean-Pierre Gazeau~
and \\ 
Jean-Louis Verger-Gaugry
}

\address{}

\maketitle

\noindent
{\bf R\'esum\'e}
On explore le spectre d'un peigne de Dirac 
pond\'er\'e support\'e par le quasicristal de Thue-Morse,
et on le caract\'erise \`a un ensemble de 
mesure nulle pr\`es, au moyen de la Conjecture de
Bombieri-Taylor, pour les pics de Bragg, et 
d'une autre conjecture que l'on appelle Conjecture de 
Aubry-Godr\`eche-Luck, pour la composante singuli\`ere continue.
La d\'ecomposition de la transform\'ee de Fourier
du peigne de Dirac pond\'er\'e est obtenue dans 
le cadre de la th\'eorie des distributions
temp\'er\'ees.
Nous montrons que l'asymptotique de l'arithmetique des
sommes $p$-rar\'efi\'ees de Thue-Morse
(Dumont; Goldstein, Kelly and Speer;
Grabner; Drmota and Skalba,...), pr\'ecis\'ement
les fonctions fractales des sommes de chiffres,  jouent
un r\^ole fondamental dans la description de la composante singuli\`ere continue
du spectre, combin\'ees \`a des r\'esultats classiques
sur les produits de Riesz de Peyri\`ere et de
M. Queff\'elec. Les lois d'\'echelle dominantes
des suites de mesures approximantes sont contr\^ol\'ees 
sur une partie de la composante singuli\`ere continue
par certaines in\'egalit\'es dans lesquelles 
le nombre de classes de diviseurs et le r\'egulateur
de corps quadratiques r\'eels interviennent.

\vspace{0.2cm}

\begin{abstract} 
The spectrum of a weighted Dirac comb on the Thue-Morse quasicrystal
is investigated, and characterized up to a measure zero set, 
by means of the 
Bombieri-Taylor conjecture, for Bragg peaks, 
and of another conjecture
that we call Aubry-Godr\`eche-Luck conjecture, 
for the singular continuous component.
The decomposition of the Fourier transform of the weighted Dirac comb is
obtained in terms of tempered distributions.
We show that the asymptotic arithmetics of the 
$p$-rarefied sums of the 
Thue-Morse sequence (Dumont; Goldstein, Kelly and Speer; 
Grabner; Drmota and Skalba,...), 
namely the fractality of sum-of-digits functions, 
play a fundamental role
in the description of the 
singular continous part of the spectrum, combined with
some classical results on
Riesz products of Peyri\`ere and M. Queff\'elec.
The dominant scaling of the sequences of approximant measures
on a part of the singular component is controlled by certain inequalities 
in which are involved the class number and the regulator
of real quadratic fields.
\end{abstract}

\vspace{1.5cm}
{\it In honor of the $60$-th birthday of Henri Cohen...}
\vspace{2cm}


\vspace{0.5cm}


\vspace{0.5cm}

\tableofcontents  
\newpage
\section{Introduction}
\label{S1}

The $\pm$ Prouhet-Thue-Morse sequence $(\eta_n)_{n \in \mathbb{N}}$
is defined by     
\begin{equation}
\label{tm}
\eta_n = (-1)^{s(n)} \, \qquad n \geq 0, 
\end{equation}
where $s(n)$ is equal to the sum of the $2$-digits
$n_0 + n_1 + n_2 + \ldots $ 
in the binary expansion of 
$n = n_0 + n_1 2 + n_2 2^2 + \ldots$
It can be viewed as a
fixed point of the substitution $1 \to 1 ~\overline{1},
\overline{1} \to \overline{1} ~1$ on the two letter alphabet
$\{\pm 1\}$, starting with $1$ ($\overline{1}$ stands for $-1$).
There exists a large literature on this sequence
\cite{allouchemendesfrance} \cite{queffelec1}.
Let $a$ and $b$ be two positive real numbers such that
$0 < b < a$. Though there exists an infinite number of 
ways of constructing a regular aperiodic
point set of the line \cite{lagarias1}
from the sequence 
$(\eta_n)_{n \in \mathbb{N}}$, we adopt the following 
definition, which seems to be fairly canonical.
We call Thue-Morse quasicrystal, denoted by
$\Lambda_{a,b}$, or simply by $\Lambda$ (without 
mentioning the parametrisation by $a$ and $b$),  
the point set
\begin{equation}
\label{tmqc}
\Lambda ~:=~ \Lambda^{+} \cup (- \Lambda^{+}) \quad \subset \mathbb{R}
\end{equation}
where $\Lambda_{a,b}^{+}$, or simply $\Lambda^{+}$, 
on $\mathbb{R}^{+}$, is equal to 
\begin{equation}
\label{tmqc}
\{0\} \cup \Bigl\{ f(n) := \sum_{0 \leq m \leq n-1} 
\bigl( \frac{1}{2} (a+b) + \frac{1}{2} (a-b) \eta_m \bigr) \mid n = 1, 2, 3, \ldots 
\Bigr\}.  
\end{equation} 
The function
$f$ defined by \eqref{tmqc} is extended to $\mathbb{Z}$ by 
symmetry: we put 
\begin{equation}
\label{tmqc1}
f(0) = 0 ~\mbox{by convention and}~ 
f(n) = -f(-n) ~\mbox{for}~ n \in \mathbb{Z}, \ n < 0.
\end{equation}
For all $n \in \mathbb{Z}$,
$|f(n+1)-f(n)|$ is equal either to $a$ 
or $b$ so that the closed (generic) intervals 
of respective lengths $a$ and $b$  
are the two prototiles of the aperiodic 
tiling of the line $\mathbb{R}$
whose $(f(n))_{n \in \mathbb{Z}}$ 
is the set of vertices. 
Though it is easy to check that
$$(a+b) \, \mathbb{Z} ~\subset~ \Lambda_{a,b}$$
and moreover that
$\Lambda_{a,b}$ is a Meyer set 
\cite{lagarias1}  
\cite{moody} \cite{vergergaugry},  
i.e. there exists a finite set $F
= \{\pm a, \pm b, \pm 2 a, \pm 2 b, \pm a \pm b\}$ such that
$$\Lambda_{a,b} ~-~ \Lambda_{a,b} ~\subset~  \Lambda_{a,b} + F,$$
the Thue-Morse quasicrystal, 
and any weighted Dirac comb on it,
is considered as a somehow myterious point set, intermediate between 
chaotic, or random, and periodic \cite{aubrygodrecheluck}
\cite{axelterauchi} \cite{bai} \cite{chengsavitmerlin}
\cite{godrecheluck1} \cite{godrecheluck2}
\cite{kolariochumraymond}
\cite{luck}
\cite{peyrierecockayneaxel}
\cite{wolnywnekvergergaugry}, and
the interest for such systems in physics 
is obvious from many viewpoints.  

In this note 
we study the spectrum of a weighted Dirac comb $\mu$ on 
the point set
$\Lambda_{a,b}$ by using arithmetic methods, more precisely by 
involving sum-of-digits fractal fonctions associated with
the rarefied sums of the Thue-Morse sequence 
(Coquet \cite{coquet}, Dumont \cite{dumont},
Gelfond \cite{gelfond}, Grabner \cite{grabner1}, Newman \cite{newman},
Goldstein, Kelly and Speer \cite{goldsteinkellyspeer},
Drmota and Skalba \cite{drmotaskalba1}
\cite{drmotaskalba2}, ...). 
For this, we hold for true two conjectures which are expressed
in terms of scaling laws of approximant measures: 
the Bombieri-Taylor Conjecture and a
conjecture that we call Aubry-Godr\`eche-Luck Conjecture
(Subsection \ref{S3.2}). 
In the language of physics, 
the spectrum measures the extent to which
 the intensity diffracted 
by $\mu$ is concentrated at 
a real number $k$ (wave vector). It  
it can be observed by the 
square modulus of the Fourier transform of $\mu$ at $\{k\}$
\cite{cowley} \cite{guinier}, or eventually 
of its autocorrelation \cite{hof} \cite{lagarias2}.

On one hand 
the spectrum of the symbolic dynamical system
associated with the Prouhet-Thue-Morse sequence 
is known to be singular continuous: 
if 
\begin{equation}
\label{TFTM}
\widehat{\eta_{n}}(k) := \sum_{j=0}^{2^n - 1} \, \eta_j \exp (-2 i \pi j k) 
\end{equation} 
denotes its Fourier transform, then
\begin{equation}
\label{rieszTM}
|\widehat{\eta_{n}}(k)|^2 = 2^n \, \prod_{j=0}^{n-1} \, (1 - \cos(2 \pi 2 ^j k ))  
~=~ 2^{2 n} \, \prod_{j=0}^{n-1} \, \sin^{2} (\pi 2 ^j k)
\end{equation}
is a Riesz product constructed on the sequence
$(2^j)_{j \geq 0}$ 
which has the property that
the sequence of measures
\begin{equation}
\label{bellesuite}
\left\{ 2^{-n} |\widehat{\eta_{n}}(k)|^2 dk
\right\}_{n \geq 0}
\end{equation}  
has a unique accumulation point, its limit, 
for the vague topology,
which is a singular continuous measure
(Peyriere \cite{peyriere} \S 4.1, 
Allouche, Mend\`es-France \cite{allouchemendesfrance} Appendix I, p. 337). 
On the other hand, 
Queff\'elec (\cite{queffelec2} \S 6.3.2.1) 
has shown that replacing
the alphabet $\{\pm 1 \}$ by
$\{0, 1\}$ leads to a new component
to the measure, which is discrete and exactly
localized at the elements of the group $\mathbb{D}_2$
of $2$-adic rational numbers in the one-dimensional torus
$\mathbb{R}/2 \pi \mathbb{Z}$, explicitly
\begin{equation}
\label{DD2}
\mathbb{D}_2 := \{2 \pi \frac{m}{2^n} \mid 0 \leq m \leq 2^n , n = 0, 1, 2, \ldots \}.
\end{equation}  
Therefore the spectrum of $\Lambda_{a,b}$, 
and more generally the Fourier transform
of a weighted Dirac comb
$\mu$ supported by $\Lambda_{a,b}$, 
is expected
to be the sum of a discrete part and a singular 
continuous part, each of them being
a function of $a$ and $b$.
In Subsection \ref{S3.1}, Theorem \ref{fourierTMtheo} 
(the proof of which is similar to the proof
of Theorem 4.1 in \cite{gazeauvergergaugry} 
for sets $\mathbb{Z}_{\beta}$ of $\beta$-integers,
with $\beta$ a quadratic unitary Pisot number) 
shows that it is the case 
in the context of tempered distributions. 

We explore the particular Dirac comb
on $\Lambda_{a,b}$ 
for which the weights are all equal to $1$ on the real positive line.
In Subsection \ref{S3.3} we deduce the Bragg component of its spectrum 
by classical results on Riesz products 
of Peyri\`ere \cite{peyriere} and Queff\'elec \cite{queffelec2} and
by the Bombieri-Taylor conjecture (Theorem \ref{braggqc}).
In Subsection \ref{S3.4} we show the deep relation 
between the $p$-rarefied sums of the Thue-Morse
sequence and the singular continuous component 
of its spectrum.  
We use the sum-of-digits 
fractal functions of the rarefaction phenomenon, 
recalled in Section \ref{S4}, 
in  agreement with the Aubry-Godr\`eche-Luck
argument (Subsection \ref{S3.2}) 
to deduce, in Subsection \ref{S5.1}, the 
singular part $\mathbb{S}$
of the spectrum, up to a zero measure
subset 
(Theorem \ref{singgg}).
In the subsequent Subsections
we explore the extinctions properties and the questions of visibility
of the singular component of the spectrum as a function of the 
sequence of finite approximants.   
This leads to 
characterize prime numbers 
$p$ for which the sequences of approximant measures at the wave vectors
$k = \frac{4 \pi}{a+b} \frac{t}{2^h p}$ have the property
of being ``size-increasing" at $k$; more generally, for some classes of  
prime numbers $p$ the Cohen-Lenstra heuristics of real quadratic fields
is required \cite{cohenlenstra1} \cite{cohenlenstra2} \cite{cohenmartinet}. 
In Section \ref{S6} we indicate how to deal with 
other weighted Dirac combs on the Thue-Morse quasicrystal
$\Lambda_{a,b}$  and
prove that the diffraction process remains somehow invariant
under the Marcinkiewicz
equivalence relation, for the
Bragg peaks, by the Bombieri-Taylor argument.

In addition to the mathematical interest which 
links diffraction spectra to arithmetics
as mentioned in the title, the present results 
represent the solutions of many experimental 
questions of physicists on Thue-Morse quasicrystals 
which remained unexplained until now.


\section{Averaging sequences of finite approximants}
\label{S2}

\begin{definition}
\label{defiexhaust}
An averaging sequence $(U_l)_{l \geq 0}$ of finite approximants
of $\Lambda_{a,b}$ is given by a closed interval $J$ whose interior
contains the origin and a strictly increasing sequence
$(\rho_{l})_{l \geq 0}$ of positive numbers such that
$$U_l ~=~ \rho_l J \cap \Lambda_{a,b}, \qquad \quad l= 0, 1, 2, \ldots$$
\end{definition}

A natural averaging sequence of such approximants 
for the Thue-Morse quasicrystal
is yielded by the sequence
$((-B_{N}) \cup B_{N})_{N \geq 0}$, where
\begin{equation}
\label{canoexhaust}
B_N ~=~ \{ x \in \Lambda_{a,b} \mid 0 \leq x \leq f(N)\} 
\qquad \quad N=0, 1, 2, \ldots
\end{equation}    
Obviously \, Card$(B_N) = 2 N + 1$ for $N=0, 1, 2, \ldots$

We have the same notions 
for a Radon measure 
$\mu = \sum_{n \in \mathbb{Z}} \, \omega(n) \delta_{f(n)}$
supported by $\Lambda_{a,b}$, where
$\omega(x)$ is a bounded complex-valued function, called weight, and
where $\delta_{f(n)}$ denotes the normalized
Dirac measure supported by the singleton
$\{f(n)\}$.
The support of the weight fonction $\omega$ is denoted
$$\mbox{supp}(\omega) := \{n \in \mathbb{Z} \mid \omega(n) \neq 0 \}.$$
Denote 
$$\Lambda_{a,b}^{\mbox{{\scriptsize supp}}(\omega)} ~=~ \bigcup_{n \in \, \mbox{{\scriptsize supp}}(\omega)} \, \{f(n)\}.$$
Then an averaging sequence $(U_l)_{l \geq 0}$ of finite approximants
of $\Lambda_{a,b}^{\mbox{{\scriptsize supp}}(\omega)}$ 
is given by a closed interval $J$ whose interior
contains the origin and a strictly increasing sequence
$(\rho_{l})_{l \geq 0}$ of positive numbers such that
$$U_l ~=~ \rho_l J \cap \Lambda_{a,b}^{\mbox{{\scriptsize supp}}(\omega)}, 
\qquad \quad l= 0, 1, 2, \ldots.$$

Let us denote 
$\mathcal{A}_{a,b}$, resp. $\mathcal{A}_{a,b}^{\mbox{{\scriptsize supp}}(\omega)}$,    
the set of the averaging sequences of finite
approximants of $\Lambda_{a,b}$, 
resp. $\Lambda_{a,b}^{\mbox{{\scriptsize supp}}(\omega)}$.
If
the affine hull of supp$(\omega)$
is $\mathbb{R}$, then 
$\mathcal{A}_{a,b} = \mathcal{A}_{a,b}^{\mbox{{\scriptsize supp}}(\omega)}$.

\section{Diffraction spectra}
\label{S3}

\subsection{Fourier transform of a weighted Dirac comb on $\Lambda_{a,b}$}  
\label{S3.1}

Let us consider the following Radon measure 
supported by the Thue-Morse quasicrystal
$\Lambda_{a,b}$:
\begin{equation}
\label{radonmu}
\mu = \sum_{n \in \mathbb{Z}} \, \omega(n) \delta_{f(n)}. 
\end{equation}
where $\omega(x)$ is a bounded complex-valued 
function (weight function). 
Denote by $\delta$ the Dirac measure, i.e.
such that $\delta(\{0\})=1$.
The measure \eqref{radonmu} is translation bounded, 
i.e. for all compact $K \subset \mathbb{R}$
there exists $\alpha_K$ such that 
$\sup_{\alpha \in \mathbb{R}} \, |\mu|(K+\alpha)
\leq \alpha_K$, and so is a tempered distribution.
Its Fourier transform 
$\widehat{\mu}$
is also a tempered distribution defined by
\begin{equation}
\label{mufourier}
\widehat{\mu}(k) = \mu(e^{- i k x }) = 
\sum_{n \in \mathbb{Z}} \, \omega(n) e^{- i k f(n) }
\end{equation}
It may or may not be a measure.

Let us now transform \eqref{mufourier}
using the general form of $f(n)$, $n \in \mathbb{Z}$, 
given by
\eqref{tmqc} and \eqref{tmqc1}. For a real 
number $x \in \mathbb{R}$, the integer part
of $x$ is denoted by $\lfloor x \rfloor$ and its fractional part by
$\{x\}$.

\begin{lemma}
\label{lem1}
If $n \in \mathbb{Z}$ is even, then $f(n) = n (a+b)/2$. If
$n \geq 1$ is odd, then 
\begin{equation}
\label{lem1eq}
f(n) ~=~ \frac{n}{2} (a+b) + \frac{a-b}{2} \eta_{n-1}.  
\end{equation}
\end{lemma}

\begin{proof}
Obvious since,  for every even integer $m \geq 0$, 
$\eta_m + \eta_{m+1} = 0$.   
\end{proof}

We deduce that $f(n)$ has the general form
$f(n) = \alpha_{1} n + \alpha_0 + \alpha_2 \{x(n)\}$
for all integers $n \geq 1$, where 
$$
\displaystyle
\begin{array}{l}
\alpha_{1} ~=~ (a+b)/2\\
\alpha_0 ~=~ -(a-b)/2\\
\alpha_2 ~=~ 2 (a-b)\\
x(n) ~=~ (1 + \eta_{n-1})/4.  
\end{array}
$$
Indeed, $x(n) \in [0, 1)$ is equal to its fractional part.
Now, for any $x \in \mathbb{R}$, the ``fractional part"
function $x \to \{x\}$ is periodic of period $1$ and 
so is the piecewise
continuous function $e^{-i k \alpha_2 \{x\}}$.  
Let 
\begin{eqnarray}
c_m (k) & = & \int_{0}^{1} \, e^{-i k \alpha_2 \{x\}} e^{- 2 \pi i m x} dx\\
& = & i \frac{e^{- i k \alpha_2} - 1}{2 \pi m + k \alpha_2}
~=~
(-1)^m e^{- i k \frac{\alpha_2}{2}} 
\mbox{sinc} 
\bigl(  
\frac{\alpha_2}{2} k + m \pi
\bigr)
\label{cmk}
\end{eqnarray}
be the coefficients of the expansion of $e^{-i k \alpha_2 \{x\}}$ in Fourier
series. In 
\eqref{cmk}, sinc denotes the cardinal sine function
sinc$(x) = \sin x / x$. Thus
\begin{equation}
\label{devpp1}
e^{-i k \alpha_2 \{x\}} ~=~ \sum_{-\infty}^{+\infty} \, c_m (k) e^{2 i \pi m x}
\end{equation}
where the convergence is punctual in the usual sense of Fourier
series for piecewise continuous functions.
Thus
$$\sum_{n \geq 1} \, \omega(n) e^{- i k f(n)} ~=~ 
\sum_{n \geq 1} \, \omega(n) e^{- i k
\bigl[
\frac{n}{2} (a+b) + \frac{1}{2} (a-b) \eta_{n-1}
\bigr]}
$$
$$=
\sum_{n \geq 1} \, \omega(n) e^{- i k
\bigl[
\frac{n}{2} (a+b) - \frac{a-b}{2}\bigr]}
\times \left( 
\sum_{m \in \mathbb{Z}} \, c_m (k) e^{2 i \pi m \frac{1 + \eta_{n-1}}{4}
}
\right)
$$
$$=
\sum_{m \in \mathbb{Z}} \sum_{n \geq 1} \,
c_m (k) \omega(n) e^{- i k
\bigl[
\frac{n}{2} (a+b) - \frac{a-b}{2}\bigr]}
\times
e^{2 i \pi m \frac{1 + \eta_{n-1}}{4}}.$$

\begin{lemma}
\label{lem2}
For every integer $n \geq 1$,
\begin{equation}
\label{lem2eq}
s(n) ~=~ \frac{(-1)^{s(n)}-1}{2} \qquad (\mbox{{\rm mod}} ~2).
\end{equation}
\end{lemma}

\begin{proof}
Obvious.  
\end{proof}

We deduce
$$\sum_{n \geq 1} \, \omega(n) e^{- i k f(n)} ~=~
\sum_{m \in \mathbb{Z}} \sum_{n \geq 1} \,
c_m (k) \omega(n) e^{- i k
\bigl[
\frac{n}{2} (a+b) - \frac{a-b}{2}\bigr]}
\times
e^{i \pi m (1 + s(n-1))},
$$
and then,
since $e^{i \pi} = -1$, 
$$=~
\sum_{m \in \mathbb{Z}} \sum_{n \geq 1} \,
\left(
c_m (k) \, (- \eta_{n-1})^{m} \, \omega(n) 
e^{ - i \pi k  
\bigl(
n -\frac{a-b}{a+b}
\bigr)
\frac{(a+b)}{2 \pi}
}\right).
$$
Therefore
$$\widehat{\mu}(k) ~=~ \sum_{n \geq 1} \, \omega(-n) e^{+ i k f(n)} + \omega(0) +
\sum_{n \geq 1} \, \omega(n) e^{- i k f(n)} 
$$
$$
=~ \omega(0) + 
\sum_{m \in \mathbb{Z}} \sum_{n \geq 1} \,
\left(
c_m (k) \, (- \eta_{n-1})^{m} \, \omega(n)
e^{ - i \pi 
\bigl(
n -\frac{a-b}{a+b}
\bigr)
k \frac{(a+b)}{2 \pi}
}
\right)$$
$$
\mbox{} \hspace{2cm}
+ \sum_{m \in \mathbb{Z}} \sum_{n \geq 1} \,
\left(
c_m (-k) \, (- \eta_{n-1})^{m} \, \omega(-n)
e^{ + i \pi 
\bigl(
n -\frac{a-b}{a+b}
\bigr)
k \frac{(a+b)}{2 \pi}
}
\right).
$$
Let us observe that $(- \eta_{n-1})^{m} = 1$ if $m$ is even and that
$(- \eta_{n-1})^{m} = - \eta_{n-1}$ if $m$ is odd, for all $n \geq 1$.
Then
$$\widehat{\mu}(k) ~=~
\omega(0) + \sum_{m \in \mathbb{Z}}
c_m (k) 
\left[
\frac{1 + (-1)^m}{2} 
\left(
\sum_{n \geq 1} \, \omega(n)
e^{ - i \pi
\bigl(
n -\frac{a-b}{a+b}
\bigr)
k \frac{(a+b)}{2 \pi}
}
\right)
\right.
\hspace{2cm} \mbox{}$$
$$\mbox{} \hspace{4cm} \left. -
\frac{1 - (-1)^m}{2}
\left(
\sum_{n \geq 1} \, \omega(n) \eta_{n-1}
e^{ - i \pi
\bigl(
n -\frac{a-b}{a+b}
\bigr)
k \frac{(a+b)}{2 \pi}
}
\right)
\right]$$
\begin{equation}
\label{basedecomp}
+ \sum_{m \in \mathbb{Z}}
c_m (-k)
\left[
\frac{1 + (-1)^m}{2}
\left(
\sum_{n \geq 1} \, \omega(-n)
e^{ + i \pi
\bigl(
n -\frac{a-b}{a+b}
\bigr)
k \frac{(a+b)}{2 \pi}
}
\right)
\right.
\hspace{2cm} \mbox{}$$
$$\mbox{} \hspace{4cm}
\left. -
\frac{1 - (-1)^m}{2}
\left(
\sum_{n \geq 1} \, \omega(-n) \eta_{n-1}
e^{ + i \pi
\bigl(
n -\frac{a-b}{a+b}
\bigr)
k \frac{(a+b)}{2 \pi}
}
\right)
\right].  
\end{equation}

Suppose 
\begin{equation}
\label{rapportrationnel}
(a-b)/(a+b) \in \mathbb{Q} 
\qquad (\Longleftrightarrow ~~\frac{a}{b} \in \mathbb{Q})
\end{equation}
and denote 
$$g_{a,b} := \gcd(a-b,a+b) \geq 1.$$

\begin{lemma}
\label{gab}
If $a, b \in \mathbb{N} \setminus \{0\}$
are such that $a/b \in \mathbb{Q}$ is an irreducible fraction, then
$g_{a,b} = 2$ if $a$ and $b$ are both odd, and 
$g_{a,b} = 1$ if $a$ is odd and $b$ is even, or 
if $a$ is even and $b$ is odd. 
\end{lemma}

\begin{proof}
Obvious.
\end{proof}

\begin{remark}
\label{rem1}
The hypothesis \eqref{rapportrationnel} is fundamental since 
$$
x ~\to~ \sum_{n \geq 1} \left(
\omega(\pm n) e^{\mp 2 i \pi (n - \frac{a-b}{a+b}) x \frac{a+b}{g_{a,b}}}\right)$$
and
$$
x ~\to~ \sum_{n \geq 1} \left(
\omega(\pm n) \eta_{n-1} e^{\mp 2 i \pi (n - \frac{a-b}{a+b}) x \frac{a+b}{g_{a,b}}} \right) 
$$
are then periodic of period $1$
and that they can be considered, 
under some assumptions (see below),  
as the derivative in the distributional sense
of a function of period 1. 
The existence of these four new functions of period $1$
would be 
impossible if \eqref{rapportrationnel} 
would not hold. 
Then if $(a-b)/(a+b) \not\in \mathbb{Q}$ 
the decomposition into a discrete part and 
a continuous part of the Fourier transform, 
which is deduced from them (see below), 
would be impossible except if 
$\widehat{\mu}(k)$ is equal to $0$.
In other terms we can say that
the condition
$(a-b)/(a+b) \not\in \mathbb{Q}$
makes inexistant
the diffraction phenomenon by the 
weighted Dirac comb \eqref{radonmu}, 
supported by the
Thue-Morse quasicrystal $\Lambda_{a,b}$; this one 
does not occur
whatever the weight function $\omega(x)$
is (no concentration of intensity at any wave vector $k$). 
\end{remark}

\begin{remark}
If $a = b > 0$ and $\omega(n) = 1$ for $n \geq 0$ and 
$\omega(n) = 0$ for $n \leq -1$, then  
the tempered distribution 
$\widehat{\mu}(k)$
is an infinite sum and is reminiscent of 
the Fourier transform \eqref{TFTM} of the
$\pm 1$ Prouet-Thue-Morse sequence in the usual way.
Let us observe that it is the sections 
of this tempered distribution
(finite sums obtained by truncating its ``tail") which
provide the Riesz products \eqref{rieszTM} 
when particular averaging sequences of length
a power of $2$ are used. We will  
approach this dependance to the choice of the
averaging sequence by the rarefaction phenomenon
(Section \ref{S4}). 
\end{remark}

Let us assume \eqref{rapportrationnel}, 
i.e. $a/b \in \mathbb{Q}$.
Then we make the following hypothesis:
\paragraph{Hypothesis  {\bf (H)}}

we suppose that
the Fourier series
$$
\sum_{n \geq 1} 
\Bigl(
\frac{\omega(n) \, e^{- 2 i \pi (n - \frac{a-b}{a+b}) x \frac{a+b}{g_{a,b}}}}
{\frac{-i \pi (a+b)}{g_{a,b}} 
\bigl(n - \frac{a-b}{a+b}\bigr)}  
\Bigr),  
\quad
\mbox{resp.} 
~\sum_{n \geq 1}
\Bigl(
\frac{\omega(-n) \, e^{+ 2 i \pi (n - \frac{a-b}{a+b}) x \frac{a+b}{g_{a,b}}}}
{\frac{+ i \pi (a+b)}{g_{a,b}}
\bigl(n - \frac{a-b}{a+b}\bigr)}  
\Bigr), 
$$
resp.  
$$
\sum_{n \geq 1}
\Bigl(
\frac{\omega(n) \eta_{n-1}\, e^{- 2 i \pi (n - \frac{a-b}{a+b}) x \frac{a+b}{g_{a,b}}}}
{\frac{-i \pi (a+b)}{g_{a,b}}
\bigl(n - \frac{a-b}{a+b}\bigr)}
\Bigr), 
\quad
\mbox{resp.}  
~\sum_{n \geq 1}
\Bigl(
\frac{\omega(-n) \eta_{n-1} \, e^{+ 2 i \pi (n - \frac{a-b}{a+b}) x \frac{a+b}{g_{a,b}}}}
{\frac{+ i \pi (a+b)}{g_{a,b}}
\bigl(n - \frac{a-b}{a+b}\bigr)}
\Bigr),
$$
converges in the punctual sense to a periodic piecewise continuous function
$$f_{\omega +}(x), ~\mbox{resp.} ~f_{\omega -}(x), ~f_{\omega \eta +}(x), ~f_{\omega \eta -}(x),$$
of period $1$, which 
has a  derivative 
$$f'_{\omega +}(x), ~\mbox{resp.} ~f'_{\omega -}(x), ~f'_{\omega \eta +}(x), ~f'_{\omega \eta -}(x),$$
continuous bounded on the open set
$$\mathbb{R} -\cup_{p} \{a_{p,+}\}, ~\mbox{resp.}  
~\mathbb{R} -\cup_{p} \{a_{p,-}\}, ~\mathbb{R} -\cup_{p} \{a_{p,\eta +}\}, 
~\mathbb{R} -\cup_{p} \{a_{p,\eta -}\},$$ 
where  
$$a_{p,+}, \mbox{resp.} ~a_{p,-}, ~a_{p,\eta +}, ~a_{p,\eta -},$$ 
are the respective discontinuity points of
$$f_{\omega +}(x), \mbox{resp.} f_{\omega -}(x), ~f_{\omega \eta +}(x), ~f_{\omega \eta -}(x).$$ 
Let 
$$\sigma_{p +} = f_{\omega +}(a_{p,+} + 0) -f_{\omega +}(a_{p,+} - 0),$$ 
$$\sigma_{p -} = f_{\omega -}(a_{p,-} + 0) -f_{\omega -}(a_{p,-} - 0),$$ 
$$\sigma_{p \eta +} = f_{\omega \eta +}(a_{p,\eta +} + 0) -f_{\omega \eta +}(a_{p,\eta +} - 0),$$
$$\sigma_{p \eta -} = f_{\omega \eta -}(a_{p,\eta -} + 0) -f_{\omega \eta -}(a_{p,\eta -} - 0),$$
be the respective jumps at $a_{p,+}, \mbox{resp.} ~a_{p,-}, ~a_{p,\eta +}, ~a_{p,\eta -},$ of 
$$f_{\omega +}(x), ~\mbox{resp.} ~f_{\omega -}(x), ~f_{\omega \eta +}(x), ~f_{\omega \eta -}(x).$$ 
Then, by a classical result (\cite{schwartz}, Chap. II, \S 2), we have the equality 
$$\sum_{n \geq 1} \Bigl(
\omega(n) e^{- 2 i \pi (n - \frac{a-b}{a+b}) \Bigl( \frac{g_{a,b}}{4 \pi} k\Bigr)
 \frac{a+b}{g_{a,b}}}
\, \Bigr) \hspace{6cm} \mbox{}$$ 
\begin{equation}
\label{schwartzum}
\mbox{} \hspace{2cm} = f'_{\omega +}\Bigl( \frac{g_{a,b}}{4 \pi} k \Bigr)  
+
\sum_{p \in \mathbb{Z}} \sigma_{p +} 
\delta
\Bigl(
\frac{g_{a,b}}{4 \pi} k - a_{p, +} \Bigr)
\end{equation}
where $f'_{\omega +}$ has to be taken in a distributional sense.
We have similar equalities for the three other summations. 
Then we are led to the following 
(formal) formula for the Fourier
transform of $\mu$: 
$$\widehat{\mu}(k) ~=~
\omega(0) + \sum_{m \in \mathbb{Z}}
c_m (k)
\left[
\frac{1 + (-1)^m}{2}
\left(
f'_{\omega +}\Bigl( \frac{g_{a,b}}{4 \pi} k \Bigr) +
\sum_{p \in \mathbb{Z}} \sigma_{p +}
\delta
\Bigl(
\frac{g_{a,b}}{4 \pi} k - a_{p, +} \Bigr)
\right)
\right.
\hspace{2cm} \mbox{}$$
$$\mbox{} \hspace{1cm} \left. -
\frac{1 - (-1)^m}{2}
\left(
f'_{\omega \eta +}\Bigl( \frac{g_{a,b}}{4 \pi} k \Bigr) +
\sum_{p \in \mathbb{Z}} \sigma_{p \eta +}
\delta
\Bigl(
\frac{g_{a,b}}{4 \pi} k - a_{p, \eta+} \Bigr)
\right)
\right]$$
\begin{equation}
\label{mukfourierseries}
+ \sum_{m \in \mathbb{Z}}
c_m (-k)
\left[
\frac{1 + (-1)^m}{2}
\left(
f'_{\omega -}\Bigl( \frac{g_{a,b}}{4 \pi} k \Bigr) +
\sum_{p \in \mathbb{Z}} \sigma_{p -}
\delta
\Bigl(
\frac{g_{a,b}}{4 \pi} k - a_{p, -} \Bigr)
\right)
\right.
\hspace{2cm} \mbox{}$$
$$\mbox{} \hspace{1cm}
\left. -
\frac{1 - (-1)^m}{2}
\left(
f'_{\omega \eta -}\Bigl( \frac{g_{a,b}}{4 \pi} k \Bigr) +
\sum_{p \in \mathbb{Z}} \sigma_{p \eta -}
\delta
\Bigl(
\frac{g_{a,b}}{4 \pi} k - a_{p, \eta-} \Bigr)
\right)
\right]
\end{equation}
within the framework of the distribution theory.

We now consider the convergence of the part of $\widehat{\mu}(k)$
which does not contain the
Dirac measures (with the jumps). 

Given $T > 0$ and the set of Fourier coefficients
$\{c_{m}(k) \mid m \in \mathbb{Z}\}$ in
\eqref{cmk}, let
$J_T$ be the space of 4-tuples
$\{g_1(k), g_2(k), g_3(k), g_4(k)\}$,
where each $g_i(k)$
is a complex valued function,
such that the series
$$\sum_{m \in \mathbb{Z}}
\left(
c_m (k)
\Bigl[
\frac{1 + (-1)^m}{2}
\left(g_1(k/T)
\right)
-
\frac{1 - (-1)^m}{2}
\left(
g_2(k/T)
\right)\Bigr]
\right.$$
\begin{equation}
\label{jtseries}
\left.
+
c_m (-k)
\Bigl[
\frac{1 + (-1)^m}{2}
\left(
g_3(k/T)
\right)
-
\frac{1 - (-1)^m}{2}
\left(
g_4(k/T)
\right)\Bigr]
\right)
\end{equation}
converges (in the punctual sense)
to a function $G(k)$ which is slowly increasing
and locally integrable in $k$: there exists $A > 0$ and $\nu > 0$
such that $|G(k)| < A |k|^{\nu}$ for
$k \to +\infty$.
As a matter of fact, for any $T > 0$, all 4-tuples of
Fourier exponentials
$e^{i \lambda k}$ are in $J_T$ for all $\lambda \in \mathbb{R}$,
by \eqref{devpp1}.

If we assume
that
$\{f'_{\omega +}, f'_{\omega -}, f'_{\omega \eta +}, f'_{\omega \eta -}\}
\in J_T$ with
$T=4 \pi/g_{a,b}$
then, assuming true the hypothesis {\bf (H)}, 
the first part of $\widehat{\mu}(k)$ 
which does not contain the Dirac measures 
defines a regular tempered distribution, say $\widehat{\mu}_r(k)$.
From \eqref{mukfourierseries}, 
we infer that if the Fourier
transform
of the measure $\mu$ can be interpreted as a measure too, then the first term
$\widehat{\mu}_r(k)$ may be interpreted as a 
``continuous part" of that measure.  

Summarizing we have proved the following result.

\begin{theorem}
\label{fourierTMtheo}
Assume $a/b \in \mathbb{Q}$ and 
let $g_{a,b} = ${\rm gcd}$(a-b,a+b)$.
Suppose that the hypothesis {\bf ({\rm H)}} 
is true and that 
$\{f'_{\omega +}, f'_{\omega -}, f'_{\omega \eta +}, f'_{\omega \eta -}\}$ 
belongs to $J_T$ with
$T=4 \pi/g_{a,b}$.  
Then \begin{itemize}
\item[(i)]
the Fourier transfom $\widehat{\mu}$ of $\mu$ 
is the sum of a pure point 
tempered distribution $\widehat{\mu}_{pp}$ 
and a regular tempered distribution
$\widehat{\mu}_r$
as follows:
$$\widehat{\mu} ~=~
\widehat{\mu}_{pp} + \widehat{\mu}_r$$
with
$\widehat{\mu}_r (k) ~=~ 
\omega(0)$
$$ +
\sum_{m \in \mathbb{Z}}
\left(
c_m (k)
\Bigl[
\frac{1 + (-1)^m}{2}
\left(f'_{\omega +}(k/T)
\right)
-
\frac{1 - (-1)^m}{2}
\left(
f'_{\omega \eta +}(k/T)
\right)\Bigr]
\right.$$
\begin{equation}
\label{mur}
\left.
+
c_m (-k)
\Bigl[
\frac{1 + (-1)^m}{2}
\left(
f'_{\omega \eta +}(k/T)
\right)
-
\frac{1 - (-1)^m}{2}
\left(
f'_{\omega \eta -}(k/T)
\right)\Bigr]
\right),  
\end{equation}
and with
$\widehat{\mu}_{pp}(k) ~=~$  
$$\sum_{m, p \in \mathbb{Z}}
\left(
c_m (k)
\Bigl[
\frac{1 + (-1)^m}{2}
\sigma_{p +}
\delta
\Bigl(
\frac{k}{T} - a_{p, +} \Bigr)
-
\frac{1 - (-1)^m}{2}
\sigma_{p -}
\delta
\Bigl(
\frac{k}{T} - a_{p, -} \Bigr)
\Bigr]
\right.$$
\begin{equation}
\label{mupp}
\left.
+
c_m (-k)
\Bigl[
\frac{1 + (-1)^m}{2}
\sigma_{p \eta +}
\delta
\Bigl(
\frac{k}{T} - a_{p, \eta +} \Bigr)
-
\frac{1 - (-1)^m}{2}
\sigma_{p \eta -}
\delta
\Bigl(
\frac{k}{T} - a_{p, \eta -} \Bigr)
\Bigr]
\right),  
\end{equation}

\item[(ii)] the pure point part $\widehat{\mu}_{pp}$ is 
supported by the union
$$T \times \bigcup_{p \in \mathbb{Z}} \, \{a_{p,+}, a_{p,-}, a_{p,\eta +}, a_{p,\eta -}\}.$$   
\end{itemize}
\end{theorem}

\subsection{Scaling behaviour of approximant measures}
\label{S3.2}

It has becoming traditional for the last two decades
to define the Bragg spectrum, i.e.
the more or less bright spots which are observed 
in a diffraction experiment
on a long-range order material (say defined by
$\chi$, as a weighted 
sum of Dirac measures localized at the atomic sites), 
as the pure-point 
component of the diffraction measure, defined by the Fourier
transform of the so-called autocorrelation
$\gamma$ of the measure $\chi$
(for the definition of $\gamma$ in the context of
aperiodic point sets, see \cite{hof}).

This calls for some assumptions:  
on the point set of diffractive sites
(whether it is a model set, a Meyer set or a 
Delone set with additional properties, etc), 
on the uniqueness of 
the autocorrelation $\gamma$ (Theorem 3.4 in \cite{hof}), 
on the existence of a limit
in the sense of Bohr-Besicovich for the averaged 
Fourier transform of finite approximants
of the measure $\chi$ (Theorems 5.1 and 5.3 in
\cite{hof}). Under these assumptions, the
Bragg component of the spectrum can be computed using
the so-called Bombieri-Taylor argument. 

Let us make precise this context in the case
of a weighted Dirac comb 
$\mu = \sum_{n \in \mathbb{Z}} \, \omega(n) \delta_{f(n)}$
on $\Lambda_{a,b}$ as \eqref{radonmu}.
In the general context of diffraction physics, 
the diffracted
intensity at $k$, per diffracting site, 
(with $a/b \in \mathbb{Q}$; see Remark \ref{rem1})
is \cite{cowley} \cite{guinier} \cite{hof} \cite{lagarias2}:
\begin{equation}
\label{intensityI}
I_{\omega}(k) ~=~ \limsup_{l \to +\infty}
\left|
\frac{1}{\mbox{card}(U_l)} \, \sum_{f(n) \in U_l} \omega(n) e^{-i k f(n)}
\right|^2 \, , 
\end{equation} 
for any averaging sequence
$(U_l)_{l \geq 0} \in \mathcal{A}_{a,b}$.

The Bombieri-Taylor argument asserts that when
\eqref{intensityI} takes
a nonzero value, for an averaging sequence
$(U_l)_{l \geq 0}$, then we have the equality
$$I_{\omega}(k) ~=~ \widehat{\gamma}_{pp}(\{k\})$$
where $\widehat{\gamma}_{pp}(\{k\})$ is 
the value  at $\{k\}$ 
of the pure-point part of the Fourier transform $\widehat{\gamma}$
of the autocorrelation $\gamma$.

The
Bombieri-Taylor argument has appeared
in \cite{bombieritaylor1} and \cite{bombieritaylor2}
without proof. Since then it has become a conjecture
(see Hof \cite{hof}) since its proof 
seems fairly difficult to establish in general
for Delone sets or Meyer sets \cite{hof}.
It is only proved for model sets
(Hof \cite{hof}), for some Meyer sets (Strungaru \cite{strungaru}), 
and in particular contexts  
(de Oliveira \cite{deoliveira}, Lenz \cite{lenz}, $\ldots$).
We do  not intend to prove it
but to reformulate it in terms of the scaling
properties of a family of approximant measures 
that we define below in a way which seems consistent with
approaches followed by many authors (Cheng, Savit and Merlin \cite{chengsavitmerlin}, 
Aubry, Godr\`eche and Luck \cite{aubrygodrecheluck} \cite{luck},
Kolar, Iochum and Raymond \cite{kolariochumraymond}).
These measures will be ``approximant" to the Fourier
transform of the autocorrelation $\gamma$ but we need not to have a precise
definition of $\gamma$, or to know whether $\gamma$ is 
unique and well-defined, etc.  
This is a simplification of the problem.
This reformulation presents the advantage to be  employed in  expressing
the other
argument, that we will call (AGL) Aubry-Godr\`eche-Luck argument. The AGL argument 
 is known to allow the computation of 
the singular continous component of the spectrum.


The following definition is natural, see \eqref{bellesuite}.

\begin{definition}
\label{defapproxlnu}
Given $\mu = \sum_{n \in \mathbb{Z}} \, \omega(n) \delta_{f(n)}$
on $\Lambda_{a,b}$ and an averaging sequence
$\mathcal{U} = (U_l)_{l \geq 0} \in \mathcal{A}_{a,b}^{{\scriptsize supp}(\omega)}$
of finite
approximants of $\Lambda_{a,b}^{{\scriptsize supp}(\omega)}$, we define
the $l$-th approximant measure $\nu_{\mathcal{U},l}(k) dk$ 
as
\begin{equation}
\label{approxlnu}
\nu_{\mathcal{U},l}(k) dk ~:=~ \frac{1}
{ \mbox{card}(
U_{l}
) } 
\, 
\left| \sum_{f(n) \in U_l} \omega(n) e^{-i k f(n)}
\right|^2 \, dk  
\end{equation}
\end{definition}

\begin{definition}
\label{argumentsBT+AGL}
Let $k \in \mathbb{R}$.
Then

(i)
{\bf (BT) Bombieri-Taylor argument}: 
$k$ belongs to the 
set of Bragg peaks of the spectrum 
of $\Lambda_{a,b}$ if
and only if
\begin{equation}
\label{braggBT}
0 < \liminf_{l \to +\infty} \frac{\nu_{\mathcal{U},l}(k) dk}{\mbox{card}(
U_{l}
) dk}
\leq 
\limsup_{l \to +\infty} \frac{\nu_{\mathcal{U},l}(k) dk}{\mbox{card}(
U_{l}
) dk}
< +\infty.  
\end{equation}
for at least one sequence of finite approximants 
$\mathcal{U}  = (U_l)_{l \geq 0} \in \mathcal{A}_{a,b}^{{\scriptsize supp}(\omega)}$,  

(ii)
{\bf (AGL) Aubry-Godr\`eche-Luck argument}:
$k$ belongs to the
singular continuous part of the spectrum of
$\Lambda_{a,b}$ if
and only if
there exists $\alpha \in (-1, 1)$, which depends only upon $k$
(not upon $\mathcal{U}$),
such that
\begin{equation}
\label{sing1AGL}
0 < \liminf_{l \to +\infty} \frac{\nu_{\mathcal{U},l}(k) dk}{\mbox{card}(
U_{l} 
)^{\alpha} dk}
\leq
\limsup_{l \to +\infty} \frac{\nu_{\mathcal{U},l}(k) dk}{\mbox{card}(
U_{l} 
)^{\alpha} dk}
< +\infty
\end{equation}
for all $\mathcal{U} ~=~ (U_l)_{l \geq 0} \in \mathcal{A}_{a,b}^{{\scriptsize supp}(\omega)}$  
such that 
$$0 \neq \liminf_{l \to +\infty} \frac{\nu_{\mathcal{U},l}(k) dk}{\mbox{card}(
U_{l}
)^{\alpha} dk},$$
or such that
\begin{equation}
\label{sing2AGL}
0 < \limsup_{l \to +\infty} \frac{\nu_{\mathcal{U},l}(k) dk}{\mbox{card}(
U_{l} 
)^{\alpha} dk}
< +\infty
\end{equation}
for all $\mathcal{U} ~=~ (U_l)_{l \geq 0} \in \mathcal{A}_{a,b}^{{\scriptsize supp}(\omega)}$
such that
$$0 ~=~ \liminf_{l \to +\infty} \frac{\nu_{\mathcal{U},l}(k) dk}{\mbox{card}(
U_{l} 
)^{\alpha} dk}.$$
\end{definition}
The fact that the exponent $\alpha$ is $< 1$ in the (AGL)-argument, 
compared to the Bombieri-Taylor Conjecture (BT),   
means that the concentration of intensity in the diffraction process
is ``less infinite"
along the singular continuous component 
than in the case of the Bragg peaks.
 
\subsection{The Bragg component}
\label{S3.3}

In the rest of the paper, we assume $a/b \in \mathbb{Q}$ so that
the diffraction phenomenon from any weighted Dirac comb on
$\Lambda_{a,b}$ can occur, by Remark \ref{rem1}. 
Then without restriction of generality
we can assume $a, b \in \mathbb{Q}$
(with still $0 < b < a$). 

Let us start by the Thue-Morse quasicrystal defined by
$$\omega(n) = 0 \quad \mbox{if}~ n \leq 0 ~\mbox{and}~
\omega(n) = 1 \quad \mbox{if}~ n \geq 1,$$
up till Section \ref{S5}, 
and report on other weighted Dirac combs
on $\Lambda_{a,b}$ in Section \ref{S6}.
Hence supp$(\omega) = \mathbb{N} \setminus \{0\}$.
Let 
\begin{equation}
\label{canon}
U_l := [0, f(l)] \cap 
\bigl(\Lambda_{a,b} \setminus \{0\}
\bigr)
\qquad \quad \mbox{for}~  
l \geq 1,
\end{equation}
so that $(U_l)_{l \geq 1} \in \mathcal{A}_{a,b}^{\mbox{{\scriptsize supp}}(\omega)}$ and 
$\mbox{card}(U_l) = l$. 
We call this sequence the canonical sequence, any other
sequence of finite approximants
in 
$\mathcal{A}_{a,b}^{\mbox{{\scriptsize supp}}(\omega)}$
being a subsequence of it.

Let $k \in \mathbb{R}$ and
$$
\kappa(k) := \sum_{m \in 2 \mathbb{Z}} \, c_{m}(k) \qquad {\rm and}\quad  
\kappa_{\eta}(k) := \sum_{m \in 2 \mathbb{Z}+1} \, c_{m}(k).$$
By \eqref{basedecomp}, we have
\begin{equation}
\label{basedecomp1}
\widehat{\mu}(k) = \kappa(k) \bigl(
\sum_{n \geq 1}
\, e^{-  \frac{i}{2} ((a+b)n - (a-b)) k
} \bigr)
-
\kappa_{\eta}(k)
\bigl(
\sum_{n \geq 1}
\, \eta_{n-1} e^{-  \frac{i}{2} ((a+b)n - (a-b)) k
}
\bigr).
\end{equation}

\begin{theorem}
\label{braggqc}
The Bragg component, denoted $\mathcal{B}_{a,b}^{(\mu)}$, 
of the spectrum of
the Thue-Morse quasicrystal given
by $\mu = \sum_{n \geq 1} \delta_{f(n)}$
on $\Lambda_{a,b}$ is exactly
the periodized normalized group of $2$-adic rational
numbers
\begin{equation}
\label{braggqcD2}
\frac{4 \pi}{a+b} \mathbb{Z} + \frac{2}{a+b} \mathbb{D}_2 ~=~ 
\left\{
\frac{4 \pi}{a+b} \bigl(r + \frac{m}{2^h}\bigr) 
\mid r \in \mathbb{Z} ~\mbox{ and}~ 0 \leq m \leq 2^h , h = 0, 1, 2, \ldots 
\right\}. 
\end{equation} 
\end{theorem}

\begin{proof}
First we write $\widehat{\mu}(k)$ in \eqref{basedecomp1}
as
$$(\kappa(k)+\kappa_{\eta}(k))
\bigl(
\sum_{n \geq 1}
\, e^{-  \frac{i}{2} ((a+b)n - (a-b)) k
} \bigr)
- 2
\kappa_{\eta}(k)
\bigl(
\sum_{n \geq 1}
\, \frac{1+\eta_{n-1}}{2} e^{-  \frac{i}{2} ((a+b)n - (a-b)) k
}
\bigr)
$$
and we take the
$l$-th approximant measures for $l = 2^h$ with 
$h = 0, 1, 2, \ldots$ for $k = \frac{4 \pi}{a+b} q$.
We have
$$\nu_{\mathcal{U},l}(k) dk ~=~ \frac{1}{2^h}  
\,
\left| 
\Bigl(\kappa\bigl(\frac{4 \pi}{a+b} q\bigr)+\kappa_{\eta}\bigl(\frac{4 \pi}{a+b} q\bigr)
\Bigr)
\Bigl(
\sum_{n= 1}^{2^h}
\, e^{-  2 i \pi (n - \frac{a-b}{a+b})q 
} \Bigr)
\right.
$$
\begin{equation}
\label{braggapprox}
\left.
- 2
\kappa_{\eta}\bigl(\frac{4 \pi}{a+b} q\bigr)
\Bigl(
\sum_{n= 1}^{2^h}
\, \frac{1+\eta_{n-1}}{2} e^{-  2 i \pi (n - \frac{a-b}{a+b})q 
}
\Bigr)
\right|^2 \, \frac{4 \pi}{a+b} dq.  
\end{equation}
We recognize the first summation as the Fourier transform of a lattice
and the second summation as the Fourier transform
of the Thue-Morse sequence on the alphabet $\{0, 1\}$ studied by
Queff\'elec \cite{queffelec2} \S 6.3.2.1 for which the spectrum is composed of
a Bragg component and a singular continuous component
(see \eqref{TFTM} to 
\eqref{DD2} in the Introduction). We now use the
(BT)-argument \eqref{braggBT} to these approximant measures, 
allowing $h$ to go to $+\infty$,  to
claim that the Bragg peaks arise from  
either the first summation or the second summation. 
The corresponding intensities can be computed
from 
$\kappa\bigl(\frac{4 \pi}{a+b} q\bigr)$
and
$\kappa_{\eta}\bigl(\frac{4 \pi}{a+b} q\bigr)
$, using \eqref{braggapprox},
passing to the limit $h \to +\infty$, 
and may take a zero value (possible extinctions).
\end{proof}

\subsection{Rarefied sums of the Thue-Morse sequence and singular continuous component of the spectrum}
\label{S3.4}

From the expression \eqref{braggapprox} of the approximant measure
with $l=2^h$, from \cite{queffelec2} \S 6.3.2.1 and the (AGL)-argument
\eqref{sing1AGL} \eqref{sing2AGL}, we infer that,  
apart from the Bragg component of the spectrum    
of
the Thue-Morse quasicrystal given
by $\mu = \sum_{n \geq 1} \delta_{f(n)}$
on $\Lambda_{a,b}$, 
the continuous component exists (i.e. is not trivial) 
and is only singular.

Assume that $k = (4 \pi q)/(a+b)$ does not belong to the Bragg component
\eqref{braggqcD2} and that $\mathcal{U}$ is the canonical sequence of finite approximants
\eqref{canon} in the following.
Then we see, when $k$ runs over the singular continuous part of the spectrum, 
using the (AGL)-argument, that the scaling behaviour 
of \eqref{braggapprox} with a power of $l = 2^h$ 
is
dictated and dominated by
the scaling behaviour of the second summation
$\sum_{n= 1}^{2^h}
\, \frac{1+\eta_{n-1}}{2} e^{-  2 i \pi (n - \frac{a-b}{a+b})q
}$ with a power of $2^h$. 

More generally, when $k$ lies 
in the singular continuous part of the spectrum, 
it is obvious to say that 
the scaling exponent $\alpha$
introduced in \eqref{sing1AGL} \eqref{sing2AGL},   
for the scaling behaviour of $\nu_{\mathcal{U},l}(k) dk$, 
takes only into account the scaling behaviour of the dominant term
$$\sum_{n= 1}^{l}
\, \eta_{n-1} e^{-  2 i \pi (n - \frac{a-b}{a+b})q
},$$ 
with a (dominant) power of $l$, if it exists, $l$ going to infinity, and if $\beta$ 
denotes this exponent, one has 
$$\alpha ~=~ 2 \beta - 1.$$
It amounts to understand
the scaling behaviour of the 
sums
\begin{equation}
\label{basicsum}
\sum_{n=1}^{l} \, \eta_{n-1} e^{-  2 i \pi n q}
\end{equation} 
with a power of $l$, when $l$ goes to infinity. 
We now precise these notions.

\begin{definition}
\label{rarefied}
Let $p \in \mathbb{N} \setminus \{0, 1, 2\}$.
The subsequences $\eta_i , ~\eta_{i+p}, ~\eta_{i+2 p }, ~\eta_{i+ 3 p}, \ldots$
are refered to as $p$-rarefied Thue-Morse sequences
for $i = 0, 1, \ldots, p-1$
($p$ is not necessarily a prime number). Their
partial sums, called $p$-rarefied sums of the Thue-Morse sequence, are denoted
\begin{equation}
\label{praresum}
S_{p, i}(n) ~=~ \sum_{
\stackrel{0 \leq j < n}{\small  j \equiv i ~(\mbox{{\scriptsize mod}} \, p)}
} \, \eta_{j} \qquad \quad n = 1, 2, \ldots
\end{equation}
\end{definition}

The following Lemma is obvious.

\begin{lemma}
\label{3possib}
Let $q \in \mathbb{R}$ and denote
$W_{q} := \{e^{-  2 i \pi n q} \mid n \in \mathbb{Z} \}$
the countable subset of the unit circle
$|z|=1$. Then, 
$$
\# \{ n \in \mathbb{Z} \mid 
e^{-  2 i \pi n q} = v\} ~=
\left\{
\begin{array}{lll}
0 & \mbox{if and only if} & v \not\in W_{q},\\
1 & \mbox{if and only if} & v \in W_{q} ~\mbox{and}~ q\not\in \mathbb{Q},\\
\infty & \mbox{if and only if} & v \in W_{q} ~\mbox{and}~ q \in \mathbb{Q}. 
\end{array}
\right.
$$
If
$q \in  \mathbb{Q}, v \in W_{q}$, 
and 
if we write
$q = \frac{t}{p}$, with
$t \in \mathbb{Z}, p \in \mathbb{N} \setminus \{0\}$, and
gcd$(t,p) = 1$, then
there exists an integer
$r \in \{0, 1, 2, \ldots, p-1\}$ such that
$$\{ n \in \mathbb{Z} \mid e^{- 2 i \pi n q} = v \} = r + p \mathbb{Z}.$$
\end{lemma}  

First, let us define the following function 
$\alpha_{l}(k)$ of $k \in \mathbb{R}$ and $l \geq 1$, 
with card$(U_l) = l$, as follows
\begin{equation}
\label{alphFUNC}
\mbox{card}(U_l)^{\alpha_{l}(k)} 
:=~
\frac{1}
{ \mbox{card}(
U_{l}
) }
\,
\left| \sum_{j=0}^{l-1} \, \eta_{j} e^{- 2 i \pi j k}
\right|^2 \, . 
\end{equation}

\begin{proposition}
\label{qnonrat}
For almost all $k \in \mathbb{R}$ 
\begin{equation}
\label{basicsum1}
\lim_{n \to +\infty} \alpha_{2^n}(k) ~=~ -1.
\end{equation} 
\end{proposition}

\begin{proof}
This is fairly classical.
We have
$$\left| \sum_{j=0}^{2^n -1} \, \eta_{j} e^{- 2 i \pi j k}
\right|^2  
~=~ 2^{2 n} \, \prod_{j=0}^{n-1} \, \sin^{2} (\pi 2^j k)
~=~
2^{2 n} \, \prod_{j=0}^{n-1} \, \sin^{2} (\pi \varphi_j)
$$
where $\varphi_j = 2^j \varphi_0, j \geq 0,$ 
with $\varphi_0 = k$. 
We deduce
$$\alpha_{2^n}(k) = 
1 + \frac{2}{n \log 2} \, \sum_ {j=0}^{n-1} \, 
\log \bigl|
\sin (\pi \varphi_{j} ) \bigr| .$$
By a theorem of Raikov \cite{raikov} \cite{kac}, for 
almost every initial value $\varphi_0 = k$, 
we can change the summation into an integral
and obtain the claim
$$\lim_{n \to +\infty}
\alpha_{2^n}(k)
~=~ 1 + \lim_{n \to +\infty}
\frac{2}{n \log 2} \, \sum_ {j=0}^{n-1} \, \log\bigl|
\sin (\pi \varphi_n ) \bigr|$$
$$= 1 + \frac{2}{\log 2} \int_{0}^{1}
\, \log\bigl|
\sin (\pi x ) \bigr| dx = -1.$$
\end{proof}

\begin{corollary}
\label{singmesNULLE}
The singular continuous component of the spectrum of
the Thue-Morse quasicrystal given
by $\mu = \sum_{n \geq 1} \delta_{f(n)}$
on $\Lambda_{a,b}$ is a subset of measure zero of $\mathbb{R}$.
\end{corollary}

\begin{proof}
Assume that $k$ belongs to the singular component.
Then, by the (AGL)-argument, there exists
$\alpha \in (-1,1)$ such that
\begin{equation}
\label{sing1AGLsing}
0 \leq \liminf_{l \to +\infty} \frac{\nu_{\mathcal{U},l}(k) dk}{\mbox{card}(
U_{l}
)^{\alpha} dk}
\leq
\limsup_{l \to +\infty} \frac{\nu_{\mathcal{U},l}(k) dk}{\mbox{card}(
U_{l}
)^{\alpha} dk}
< +\infty  
\end{equation}
for all $\mathcal{U} ~=~ (U_l)_{l \geq 0} \in \mathcal{A}_{a,b}^{{\scriptsize supp}(\omega)}$
such that
$$0 \neq \limsup_{l \to +\infty} \frac{\nu_{\mathcal{U},l}(k) dk}{\mbox{card}(
U_{l}
)^{\alpha} dk}.$$
Here, with
the particular subsequence of finite approximants
$(U_{2^n})$,
by \eqref{braggapprox} and Proposition \ref{qnonrat}, 
we have    
$$\limsup_{n \to +\infty} 
\frac{
\nu_{\mathcal{U},2^n}(k) dk}{
\mbox{card}(
U_{2^n}
)^{\alpha} dk}   ~=~ 
\bigl|
\kappa_{\eta}(k)\bigr|^2 \times \limsup_{n \to +\infty} 
\frac{
\mbox{card}(U_{2^n}
)^{\alpha_{2^n}(k)}
}{\mbox{card}(
U_{2^n}
)^{\alpha} dk}
 dk ~=~ 0$$
for almost all $k$, since 
$\lim_{n \to \infty} (-\alpha + \alpha_{2^n}(k)) = -\alpha -1 < 0$.
By the block structure of 
the Prouhet-Thue-Morse sequence
we deduce that 
$$\limsup_{l \to +\infty} \frac{
\nu_{\mathcal{U},l}(k) dk}{
\mbox{card}(
U_{l}
)^{\alpha} dk} ~=~ 0$$
holds for all 
$\mathcal{U} ~=~ (U_l)_{l \geq 0} \in \mathcal{A}_{a,b}^{{\scriptsize supp}(\omega)}$
in a similar way.  
This is impossible, except possibly for the $k$s which lie 
in a set of measure zero 
for which Proposition \ref{qnonrat} does not hold.
\end{proof}

Using the (AGL)-argument 
in a careful study of $\frac{4 \pi}{a+b} \mathbb{Q}$,
and invoking Corollary \ref{singmesNULLE}, 
we will deduce in Section \ref{S5}
that this singular continuous component
is mostly $$\frac{4 \pi}{a+b} \mathbb{Q},$$
perhaps up to a subset of 
$\mathbb{R} \setminus \frac{4 \pi}{a+b} \mathbb{Q}$ of measure zero.

Now let $q = \frac{t}{p} \in \mathbb{Q}$, with
$t \in \mathbb{Z}, p \in \mathbb{N} \setminus \{0\}$, and
gcd$(t,p) = 1$. Assume that 
$$\frac{4 \pi} {a+b} q \not\in \mathcal{B}^{(\mu)}_{a,b}.$$
This implies that $p \neq 1$ and that $p$ is not  
a power of $2$.  
For these values of $p$, by
Lemma \ref{3possib}, we have, for any integer $N \geq 1$, 
\begin{equation}
\label{qrat}
\frac{1}{N p + 1} 
\left| \sum_{n=1}^{N p + 1} \, \eta_{n-1} e^{ - 2 i  \pi n \frac{t}{p} }
\right|^2 
~=~
\frac{1}{N p + 1}
\left| \sum_{j=0}^{p-1} S_{p,j}(N p) e^{ - 2 i  \pi \frac{j t}{p} } 
\right|^2 .
\end{equation}

The equality \eqref{qrat} is the key relation for
reducing the problem to the asymptotic behaviour of the
$p$-rarefied sums $S_{p,j}(N p)$ when $N$ goes to infinity.
We will continue in Section \ref{S5} the characterization
of the singular component of the spectrum, but before we need 
to recall some basic facts about the 
$p$-rarefied sums of the Thue-Morse
sequence (in Section \ref{S4}) when $p$ is a prime number, and
develop the non-prime case
as in Grabner \cite{grabner1} \cite{grabner2}.

\section{The rarefaction phenomenon}
\label{S4}

We will  
follow Goldstein, Kelly and Speer \cite{goldsteinkellyspeer}
(which originates in Dumont \cite{dumont}) for the fractal structure 
of the $p$-rarefied sums of the Thue-Morse sequence.  
In their paper $p$ is always an odd prime number whereas we need
to consider the case in which $p$ is an integer $\geq 3$ 
which is not necessarily prime (by \eqref{qrat}), 
and which is not a power of $2$.  

In 1969 Newman \cite{newman} proved a remarkable conjecture of Moser 
which asserts that
$$S_{3, 0}(n) > 0 \qquad \quad \mbox{for any}~ n \geq 1,$$
namely that the $3$-rarefied sequence of the Thue-Morse
sequence contains a
dominant proportion of ones.
More precisely 
Newman proved, for all $n \geq 1$,
\begin{equation}
\label{newmanineq}
\frac{3^{-\beta}}{20} ~<~ 
\frac{S_{3, 0}(n)}{n^{\beta}} ~<~ 5 \, . \, 3^{-\beta}
\qquad \mbox{with}~ \beta = \frac{\log 3}{\log 4},
\end{equation}
and that the limit $\lim_{n \to +\infty} S_{3, 0}(n)/n^{\beta}$
does not exist. 

Then Coquet \cite{coquet}
improved \eqref{newmanineq} by the more precise statement
\begin{equation}
\label{coqueteq}
S_{3, 0}(n) = n^{\frac{\log 3}{\log 4}} \, . \, \psi_{3,0}\left(
\frac{\log n}{\log 4}
\right)
+ \frac{\epsilon_{3,0}(n)}{3}  
\end{equation}
where $\epsilon_{3,0}(n) \in \{0, \pm 1\}$
and
where $\psi_{3,0}(x)$ is a continuous nowhere differentiable
function of period 1 
which assumes all its values in the closed
interval
$$[\inf_{x \in  [0,1]} \psi_{3,0}(x) , \sup_{x \in  [0,1]} \psi_{3,0}(x)] ~=~ 
\left[\, \liminf \frac{S_{3, 0}(n)}{n^{\frac{\log 3}{\log 4}}} , 
\limsup \frac{S_{3, 0}(n)}{n^{\frac{\log 3}{\log 4}}} \, \right].$$ 
The latter is explicitly given by 
$$
 \left[ \left(\frac{1}{3}\right)^{\frac{\log 3}{\log 4}} \frac{2 \sqrt{3}}{3} , \frac{55}{3} 
\left(
\frac{1}{65}\right)^{\frac{\log 3}{\log 4}} 
\right].$$
Grabner \cite{grabner1} obtained in 1993 the Newman-type strict inequality 
for
$p = 5$ 
$$S_{5, 0}(n) > 0 
\qquad \quad \mbox{for any}~ n \geq 1,$$ 
and a Coquet-type fractal description
of $S_{5, 0}(n)$ as
$$S_{5, 0}(n) = n^{\frac{\log 5}{\log 16}} \Phi_{5,0}
\left(\frac{\log n}{\log 16}\right)
+ \frac{\epsilon_{5,0}(n)}{5}
\qquad \quad \mbox{for all}~ n \geq 1$$
where $\Phi_{5,0}(x)$ is also a continuous nowhere differentiable
function of period 1. We refer to \cite{grabner1} for the details.
Similar 
results were found by 
Grabner, Herendi and Tichy 
for 
$p = 17$ \cite{grabnerherenditichy}.

In 1992 Goldstein, Kelly and Speer \cite{goldsteinkellyspeer}
proposed a general matrix approach to deal simultaneously with all the 
sums $S_{p,j}(n), \, 
j \in \{0, 1, 2, \ldots, p-1\}$. They proved that
the Coquet-type expressions are very general, at least when $p$ is an odd
prime number.
They obtained, for all odd prime numbers $p$ and all $n \in \mathbb{N}$,
the existence of 
\begin{itemize}
\item
$p$ continuous nowhere differentiable functions $\psi_{p,j}(x) , j = 0, 1, 2, \ldots, p-1$,
of period $1$,   
\item an exponent $\beta = \frac{\log \lambda_1}{s \log 2}$ which remarkably
depends only upon $p$ and not on $j$,  
\item $p$ error terms $E_{p,j}(n)$ 
which are bounded above uniformly, for some constant $C > 0$,
for all $n \in \mathbb{N}$ and all $j \in \{0, 1, \ldots, p-1\}$, 
in the sense
$$|E_{p,j}(n)| \leq C n^{\beta_1}\qquad \mbox{with}~ \beta_1 = 
\left\{
\begin{array}{ll}
\frac{\log \lambda_2}{s \log 2} & \mbox{for}~ \lambda_2 > 1\\
0 & \mbox{for}~ \lambda_2  < 1
\end{array}
\right.
$$
\end{itemize}
\noindent
such that
\begin{equation}
\label{gkspeer}
S_{p,j}(n) ~=~ n^{\beta} \psi_{p,j}\left(\frac{\log n}{r \, s \, \log 2}\right)
+ E_{p,j}(n),
\end{equation}
where 
\begin{itemize}
  \item $r \in \{1, 2, 4\}$ is an integer which remarkably depends upon
$p$ and not on $j$,
  \item   $s$ is the order of $2$ in the group 
$\bigl(\mathbb{Z}/p \mathbb{Z}\bigr)^{*}$ 
of the invertible elements of the field
$\mathbb{Z}/p \mathbb{Z}$ (therefore depends upon
$p$ and not of $j$),  
  \item the real numbers 
$$\lambda_1 > \lambda_2 > \ldots \geq 0$$
are the moduli of the eigenvalues, in decreasing order, 
of the $p \times p$
matrix $$M = \bigl( S_{p, i-j}(2^s)\bigr)_{0 \leq i, j \leq p-1}$$
with integral coefficients.
\end{itemize}

\begin{proposition}
Let {\bf S}$(n) := (S_{p, 0}(n), S_{p, 1}(n), \ldots, S_{p, p-1}(n))^{t}$ denote the vector of 
$\mathbb{R}^{p}$
with integer entries (~$\mbox{}^{t}$ for transposition).
Then
\begin{equation}
\label{equajuste}
\mbox{{\bf S}}(2^s n) ~=~ M \, \mbox{{\bf S}}(n).  
\end{equation}
\end{proposition} 

\begin{proof}
\cite{goldsteinkellyspeer} p. 3.
\end{proof}

These authors obtain the fractal functions
$\psi_{p,j}(x)$ by vectorial interpolation of 
\eqref{equajuste} by constructing a fractal function
{\bf F}$(x) := (F_{0}(x), F_{1}(x), \ldots, F_{p-1}(x))^{t}$
which obeys the self-similar property
\begin{equation}
\label{equaapproxF}
\mbox{{\bf F}}(2^s x) ~=~ M \, \mbox{{\bf F}}(x).
\end{equation}
on $[0, +\infty)$ and is such that
{\bf F} = {\bf S}
on $\mathbb{N}$.

\begin{proposition}
Let $< 2 >$ denote the subgroup of
$(\mathbb{Z}/ p \mathbb{Z})^{*}$
generated by $2$ and let $a < 2 >$ any of its cosets.
Then the eigenvalues of $M$ are
\begin{equation}
\label{ptop}
\xi_a = (-2 i )^{s} \, \prod_{j \in a < 2 >} \, \sin(2 \pi j/p)
\end{equation} 
and satisfy
\begin{equation}
\label{pchia}
\prod_{a} \, \xi_a = p
\end{equation}
\end{proposition} 

\begin{proof}
Proposition 3.3 in \cite{goldsteinkellyspeer}.
\end{proof}

As a corollary of \eqref{pchia}
\begin{equation}
\lambda_1 > 1 \qquad \quad \lambda_1 \leq 2^s
\end{equation}
always hold, 
the second largest $\lambda_2$ among the magnitudes
of eigenvalues of $M$ can be $< 1$ or  
$> 1$, and the first and second magnitudes of eigenvalues
$\lambda_1$ and $\lambda_2$ can be explicitly computed for some classes
of prime numbers $p$.
These authors proved that the asymptotic growth of the summation
$S_{p,j}(n)$ is of the order of $n^{\beta}$
in the sense that
\begin{equation}
\label{gkspeerLIM}
- \infty ~<~ \liminf_{n \to \infty} \frac{S_{p,j}(n)}{n^{\beta}} 
~<~ \limsup_{n \to \infty} \frac{S_{p,j}(n)}{n^{\beta}}~<~ + \infty
\end{equation}
holds, for all odd prime number $p$ and all $j = 0, 1, \ldots, p-1$.
Let us observe that the $p$ error terms $E_{p,j}(n)$ 
become negligible at large $n$ in
\eqref{gkspeer} since $\beta_1 < \beta$, and 
therefore have no influence on the computation of
the lower and upper bounds in \eqref{gkspeerLIM}.

Several authors have investigated the bounds
liminf and limsup in \eqref{gkspeerLIM} and 
how the region between these bounds was attained by the
continous fractal functions $\psi_{p,j}(x)$: 
Coquet \cite{coquet}, Dumont
\cite{dumont}, Drmota and Skalba \cite{drmotaskalba1} \cite{drmotaskalba2},  
Grabner \cite{grabner1}, Grabner, Herendi and Tichy \cite{grabnerherenditichy}.  
It appears that it is a non-trivial problem to decide whether the continuous
function $\psi_{p,j}(x)$ has a zero or not. 
The only known examples where $\psi_{p,0}(x)$ has no zero
are $p = 3^k 5^l$ (\cite{grabner1}) and
$p = 17$ (\cite{grabnerherenditichy}).
Dumont \cite{dumont} has shown that
$\psi_{3,0}(x)$ and $\psi_{3,1}(x)$ have no zero, but
that $\psi_{3,2}(x)$ has a zero.
By the same method he proved that
for all prime numbers $p \equiv 3 ~($mod ~$4) $ 
and $j \in \{0, 1, \ldots, p-1\}$ such that 
the order of $2$ in the multiplicative group
$\bigl(\mathbb{Z}/p \mathbb{Z}\bigr)^{*}$ is $(p-1)/2$, then 
$$\psi_{p,j}(x) ~\mbox{has a zero}.$$  
Let us observe \cite{drmotaskalba2} that the assertion
that $\psi_{p,j}(x)$ has no zero is more or less equivalent to
the Newman-type inequality
$$S_{p,j}(n) > 0 \quad \mbox{for almost all}~ n ~(\mbox{or} ~~S_{p,j}(n) < 0 \quad \mbox{for almost all}~ n)$$ 
where ``almost all" means ``all but finitely many". In \cite{drmotaskalba2}
Drmota and Skalba prove (here $p$ denotes an arbitrary integer $\geq 3$ and $N \geq 1$ an integer)
$$p ~\mbox{divisible by} ~3 ~\mbox{or}~ p= 4^N + 1 
~\Longrightarrow~ S_{p,j}(n) > 0 \quad \mbox{for almost all}~ n.$$
They show that the only prime numbers $p \leq 1000$ which satisfy 
$S_{p,j}(n) > 0$ for almost all $n$ are
$p = 3, 5, 17, 43, 257, 683$ and give an asymptotic formula for their
distribution at infinity.

We now 
mention some results on the asymptotics of the sums
$S_{p, i}(n)$
when $p \geq 3$ is not a prime number.

\begin{proposition}
\label{grabneruno}
Let $1 \leq r_1, r_2$ be integers. 
The asymptotic behaviour  of 
the sum
$S_{p, 0}(n)$
with $p = 3^{r_1} 5^{r_2}$  
is given by
\begin{equation}
\label{grab35}
S_{p, 0}(p \, N) ~=~ \frac{1}{5^{r_2}} S_{3^{r_1}, 0}(p \, N) +
\frac{1}{3^{r_1}} S_{5^{r_2}, 0}(p \, N) ~+~ O(\log N)\, , 
\qquad N \geq 1,
\end{equation}
where
$$
S_{3^{r_1}, 0}(p \, N) ~=~ \frac{1}{3^{r_1} - 1} \, (p \, N)^{\alpha} \, F_{0}\bigl(\log_4 (p \, N/3) \bigr)
+
\left(\frac{p \, N}{3^{r_1}}\right)^{\alpha / 3} F_{1}\bigl(\frac{1}{3} \log_4 (\frac{p \, N}{3^{r_1}}) \bigr)
+  \ldots $$
\begin{equation}
+ \left(\frac{p \, N}{3^{r_1}}\right)^{\alpha / (3^{r_1 - 1})} F_{r_1 - 1}\bigl(\frac{1}{3^{r_1 - 1}} \log_4 (\frac{p \, N}{3^{r_1}}) \bigr)     
+ \frac{\epsilon_{3^{r_1}}\left(p \, N/ 3^{r_1}\right)}{3^{r_1}},
\end{equation}
$$
S_{5^{r_2}, 0}(p \, N) ~=~ \frac{1}{5^{r_2} - 1} \, (p \, N)^{\beta} \, G_{0}\bigl(\log_{16} (p \, N/5) \bigr)
+
\left(\frac{p \, N}{5^{r_2}}\right)^{\beta / 5} G_{1}\bigl(\frac{1}{5} \log_{16} (\frac{p \, N}{5^{r_2}}) \bigr)
+  \ldots $$
\begin{equation}
+ \left(\frac{p \, N}{5^{r_2}}\right)^{\beta / (5^{r_2 - 1})} G_{r_2 - 1}\bigl(\frac{1}{5^{r_2 - 1}} \log_{16} (\frac{p \, N}{5^{r_2}}) \bigr)
+ \frac{
\epsilon_{5^{r_2}}\left(p \, N/ 5^{r_2}\right) }{5^{r_2}},
\end{equation}
where $F_0, F_1, \ldots, F_{r_1 - 1}, G_0, G_1, \ldots, G_{r_2 -1}$ are continuous nowhere differentiable
functions 
of period $1$,
with $\epsilon_{3^{r_1}}\left(p \, N/ 3^{r_1}\right),
\epsilon_{5^{r_2}}\left(p \, N/ 5^{r_2}\right) \in   
\{0, \pm 1\}$,
and
$$\alpha = \frac{\log 3}{2 \log 2}, \qquad \quad \beta = \frac{\log 5}{4 \log 2}.$$
\end{proposition}

\begin{proof}
\cite{grabner1} pp. 40--41.
\end{proof}

\begin{remark}
\label{grabnerremark}
Since $\alpha = \frac{\log 3}{2 \log 2} > \beta = \frac{\log 5}{4 \log 2}$ and
that $\frac{\log 3}{2 \log 2}$ is the scaling exponent relative to
$S_{3, 0}(3 \, N)$, Proposition \ref{grabneruno} shows that
the scaling behaviour
of $S_{p, 0}(p \, N)$ with $p = 3^{r_1} 5^{r_2}$ is dominated 
by 
that of $S_{3, 0}(3 \, N)$, which is 
the rarefied sum of the Thue-Morse sequence relative to the prime number $3$; 
in this respect this prime number $3$ can be qualified as ``dominant"
in the prime factor decomposition of
$p = 3^{r_1} 5^{r_2}$.
Hence Proposition \ref{grabneruno} tends to show  
in general that the computation of the dominant scaling exponent   
of $S_{p, i}(p \, N)$, with $p = p_{1}^{i_1} p_{2}^{i_2} \ldots p_{m}^{i_m}$, 
its prime factor decomposition (with  $p_1 \neq 2$), $ 0 \leq i \leq p-1$,
amounts to
computing the scaling exponent relative to only one sum $S_{p_{j}, i}(n)$
where $p_j$ is the ``dominant" prime number in
the decomposition $p = p_{1}^{i_1} p_{2}^{i_2} \ldots p_{m}^{i_m}$.
General results are still missing.

\end{remark}

\section{The singular continous component}
\label{S5}

\subsection{Main Theorem}
\label{S5.1}

\begin{theorem}
\label{singgg}
The singular continous component
of the spectrum of
the Thue-Morse quasicrystal given
by $\mu = \sum_{n \geq 1} \delta_{f(n)}$
on $\Lambda_{a,b}$, 
is the set
$$\mathbb{S}:=\Bigl\{ \frac{4 \pi}{a+b} \frac{t}{2^h \, p} \mid t \in \mathbb{Z}, 
\, p ~\mbox{odd integer}~ \geq 3, \, h \geq 0, \, \mbox{gcd}(t,2^h p) = 1,\Bigr.$$
$$\mbox{} \hspace{8cm} \Bigl. \kappa_{\eta}\Bigl(
\frac{4 \pi}{a+b} \frac{t}{2^h \, p}\Bigr) \neq 0 \Bigr\}$$  
up to a subset of measure zero of 
$\mathbb{R} \setminus \frac{4 \pi}{a+b} \mathbb{Q}$.
\end{theorem}

\begin{proof}
We will only consider that $k \in \frac{4 \pi}{a+b} \mathbb{Q}$ 
and leave apart a possible subset of measure zero in
$\mathbb{R} \setminus \frac{4 \pi}{a+b} \mathbb{Q}$
by Corollary \ref{singmesNULLE}.
There are two cases: either (i) $k=
\frac{4 \pi}{a+b} \frac{t}{p}$ with $t \in \mathbb{Z}$,
$p \geq 3$ an odd integer, gcd$(t,p) =1$, or
(ii) $k=
\frac{4 \pi}{a+b} \frac{t}{2^h p}$ with $t \in \mathbb{Z}$,
$p \geq 3$ an odd integer, $h \geq 1$, gcd$(t,2^h p) =1$.

(i) Let us consider the first case.
Let us assume first that $p$ is an odd prime number.
 
\begin{definition}
\label{rarefactionpolyg}
Let $\xi_{k} := e^{- 2 i \pi \frac{t}{p}}$.
The
compact subset of the complex plane
\begin{equation}
\label{rarefdom}
\mbox{R}(k) := \Bigl\{\, \sum_{j=0}^{p-1} \, y_j \, \xi_{k}^{j} \, \Bigr\}
\end{equation}
where $ y_j$ runs over the closed interval 
$\left[
\inf_{x \in [0,1]} \psi_{p,j}(x), \sup_{x \in [0,1]} \psi_{p,j}(x)  
\right]
$, 
for $j = 0, 1, \ldots, p-1$,
is called the {\em rarefaction domain} 
at (the wave vector) $k$.
\end{definition}
The rarefaction domain R$(k)$ is obviously facetted and
may be a polygon. It may or may not contain the origin,
or (facetted) holes.
It belongs to the linear span of the powers of
the root of unity $\xi_k$ 
and inherits the properties of the fractal functions $\psi_{p,j}(x)$,
$j=0, 1, \ldots, p-1$, which are bounded. 
It is not reduced to a single point since,
for all $j=0,1, \ldots, p-1$, the two bounds
$\inf_{x \in [0,1]} \psi_{p,j}(x)$  and
$\sup_{x \in [0,1]} \psi_{p,j}(x)$ are distinct by construction.

We now show that the (AGL)-argument can be invoked, with 
suitable values of $\alpha=\alpha(k)$, for all such elements $k$. 
For $\mathcal{U}$ the canonical
sequence \eqref{canon}, recall 
the expression of the $l$-th approximant measure, with $l \geq 1$, 
$$\nu_{\mathcal{U},l}(k) dk
~=
\frac{1}{l}
\,
\left|
\Bigl(\kappa\bigl(\frac{4 \pi}{a+b} \frac{t}{p}\bigr)+\kappa_{\eta}\bigl(\frac{4 \pi}{a+b} \frac{t}{p}\bigr)
\Bigr)
\Bigl(
\sum_{n= 1}^{l}
\, e^{-  2 i \pi (n - \frac{a-b}{a+b}) \frac{t}{p}
} \Bigr)
\right.
$$
\begin{equation}
\label{nuapproxlll}
\left.
- 2
\kappa_{\eta}\bigl(\frac{4 \pi}{a+b} \frac{t}{p} \bigr)
\Bigl(
\sum_{n= 1}^{l}
\, \frac{1+\eta_{n-1}}{2} e^{-  2 i \pi (n - \frac{a-b}{a+b})\frac{t}{p} 
}
\Bigr)
\right|^2 \, dk.
\end{equation}
The rarefaction domain R$(k)$ is related to the
$l$-th approximant measure
$\nu_{\mathcal{U},l}(k) dk$, for  
$l = N p +1$ and $N \geq 1$ a positive integer,
as follows:
by \eqref{qrat}, \eqref{gkspeer} and
\eqref{rarefdom} we have 
$$\frac{1}{N p + 1}
\left| \sum_{n=1}^{N p + 1} \, \eta_{n-1} e^{ - 2 i  \pi n \frac{t}{p} }
\right|^2
~=~
\frac{1}{N p + 1}
\left| \sum_{j=0}^{p-1} S_{p,j}(N p) e^{ - 2 i  \pi \frac{j t}{p} }
\right|^2 .
$$
\begin{equation}
\label{qratnew}
~=~ 
\frac{1}{N p + 1} \bigl( N p \bigr)^{2 \beta}
\Bigl| z + \mathcal{O}\left( \bigl( N p \bigr)^{\beta_1 - \beta} \right) \Bigr|^2
\end{equation}
where $z \in $\, R$(k)$, $s \geq 1$ is the smallest integer
such that $2^s \equiv 1 \,($mod~$ p)$, where $\lambda_1$
and $\lambda_2 ~(\, < \lambda_1)$ are the first and second 
magnitudes of eigenvalues
of the matrix $M$ (see Section \ref{S4}), 
where $\beta = \frac{\log \lambda_1}{s \log 2}$
and where $\beta_1 = \frac{\log \lambda_2}{s \log 2}$ if $\lambda_2 > 1$, and
$\beta_1 = 0$ if $\lambda_2 < 1$. 
These quantities are relative to the $p$-rarefied 
sums of the Thue-Morse sequence.
Since $k$ does not belong to the Bragg component
$\mathcal{B}_{a,b}^{(\mu)}$
the second summation in \eqref{nuapproxlll}
is dominant in the asymptotic
behaviour of $\nu_{\mathcal{U},l}(k) dk$.
Therefore
\begin{equation}
\label{aglfinal}
0 \leq \liminf_{N \to +\infty} \frac{\nu_{\mathcal{U},N p + 1}(k) dk}{\mbox{card}(
U_{N p + 1}
)^{\alpha} dk}
<
\limsup_{N \to +\infty} \frac{\nu_{\mathcal{U},N p + 1}(k) dk}{\mbox{card}(
U_{N p + 1}
)^{\alpha} dk}
< +\infty
\end{equation}
with 
\begin{equation}
\label{alphamieux}
\alpha ~=~ 2 \beta - 1 ~=~ 2 \, \frac{\log \lambda_1}{s \log 2} - 1, 
\end{equation}
and
\begin{equation}
\label{mieux1}
\liminf_{N \to +\infty} \frac{\nu_{\mathcal{U},N p + 1}(k) dk}{\mbox{card}(
U_{N p + 1}
)^{\alpha} dk}
~=~
\left|\kappa_{\eta}\Bigl(\frac{4 \pi}{a+b} \frac{t}{p}\Bigr)\right|^2
\times \inf_{z \in \mbox{{\scriptsize R}}(k)} \, |z|^2 \, , 
\end{equation}
\begin{equation}
\label{mieux2}
\limsup_{N \to +\infty} \frac{\nu_{\mathcal{U},N p + 1}(k) dk}{\mbox{card}(
U_{N p + 1}
)^{\alpha} dk}
~=~
\left|\kappa_{\eta}\Bigl(\frac{4 \pi}{a+b} \frac{t}{p}\Bigr)\right|^2
\times \sup_{z \in \mbox{{\scriptsize R}}(k)} \, |z|^2 \, > 0.
\end{equation} 
It is easy to check that \eqref{aglfinal} 
holds for any other subsequence
of $\mathcal{U}$ with the same exponent $\alpha = 2 \beta -1$.
Let us notice that $1 < \lambda_1 < 2^s$ by 
\eqref{ptop}, which implies that $-1 < \alpha < 1$. Then we deduce the claim  
from the (AGL)-argument.    
\newline

General case: if now $p \geq 3$ is an odd integer, Cheng, Savit and Merlin 
(\cite{chengsavitmerlin} II.B.4) have shown that a scaling exponent
always exists. By Remark \ref{grabnerremark} this scaling exponent 
seems to be
equal, in the rarefaction phenomenon, 
to the scaling exponent of the Thue-Morse rarefied sequence 
relative to  the ``dominant" prime number in the prime number decomposition
of $p$. However, no general result is known to the authors on this subject, except the case
$p= 3^{r_1} 5^{r_2}$ with $r_1, r_2 \geq 1$ (i.e. Proposition \ref{grabneruno}
due to Grabner \cite{grabner1}).
We now proceed as in the previous case.
\newline

(ii) the second case can be deduced from the first one (i) 
as follows, now with $l = 2^n$, $n \geq 1$ : 
$$\frac{1}{2^n}
\left| \sum_{j=1}^{2^n} \, \eta_{j-1} e^{ - 2 i  \pi j \frac{t}{2^h p} }
\right|^2
~=~
2^{n} \, \prod_{j=0}^{n-1} \, \sin^{2} (\pi 2^j \frac{t}{2^h p})
$$
$$~=~
2^{h} \, \prod_{j=0}^{h-1} \, \sin^{2} (\pi 2^j \frac{t}{2^h p}) \times
2^{n-h} \, \prod_{j=0}^{n-h-1} \, \sin^{2} (\pi 2^j \frac{t}{p})$$
\begin{equation}
\label{reducenmoinsh}
=~
2^{h} \, \prod_{j=0}^{h-1} \, \sin^{2} (\pi 2^j \frac{t}{2^h p}) \times
\frac{1}{2^{n-h}}
\left| \sum_{j=1}^{2^{n-h}} \, \eta_{j-1} e^{ - 2 i  \pi j \frac{t}{p} }
\right|^2.  
\end{equation}
We have $\prod_{j=0}^{h-1} \, \sin^{2} (\pi 2^j \frac{t}{2^h p}) \neq 0$.  
Then, by \eqref{reducenmoinsh},  
we can apply the asymptotic 
laws relative to the $p$-rarefied sums of the Thue-morse sequence as  in (i).
Therefore
\begin{equation}
\label{aglfinalii}
0 \leq \liminf_{n \to +\infty} \frac{\nu_{\mathcal{U},2^n}(k) dk}{\mbox{card}(
U_{2^n}
)^{\alpha} dk}
<
\limsup_{n \to +\infty} \frac{\nu_{\mathcal{U},2^n}(k) dk}{\mbox{card}(
U_{2^n}
)^{\alpha} dk}
< +\infty
\end{equation}
with
\begin{equation}
\label{alphamieuxii}
\alpha ~=~ \alpha(k) ~=~ \alpha(2^h k) ~=~ \alpha(\frac{4 \pi}{a+b} \frac{t}{p})
~=~ 2 \beta - 1 ~=~ 2 \, \frac{\log \lambda_1}{s \log 2} - 1,   
\end{equation}
 and
\begin{equation}
\label{mieux1ii}
\liminf_{n \to +\infty} \frac{\nu_{\mathcal{U},2^n}(k) dk}{\mbox{card}(
U_{2^n}
)^{\alpha} dk}
=
\left|\kappa_{\eta}(k)\right|^2  \times  
2^{h} \, \prod_{j=0}^{h-1} \, \sin^{2} (\pi 2^j \frac{t}{2^h p})
\times \inf_{z \in \mbox{{\scriptsize R}}(2^h k)} \, |z|^2 \, ,
\end{equation}
\begin{equation}
\label{mieux2ii}
\limsup_{n \to +\infty} \frac{\nu_{\mathcal{U},2^n}(k) dk}{\mbox{card}(
U_{2^n}
)^{\alpha} dk}
=
\left|\kappa_{\eta}(k)\right|^2 \times 
2^{h} \, \prod_{j=0}^{h-1} \, \sin^{2} (\pi 2^j \frac{t}{2^h p})  
\times \sup_{z \in \mbox{{\scriptsize R}}(2^h k)} \, |z|^2.
\end{equation}
The inequalities \eqref{aglfinalii}
hold for any other subsequence
of $\mathcal{U}$ with the same exponent $\alpha = 2 \beta -1$.
The (AGL)-argument is now invoked and gives the claim.
\end{proof}

\subsection{Extinction properties}
\label{S5.2}

We say that a subsequence
$\mathcal{V} = (V_i)_{i \geq 0}$ with
$V_i = [0, f(N_i)] \cap \bigl(\Lambda_{a,b} \setminus \{0\}\bigr)$,
of the canonical sequence $\mathcal{U}$ \eqref{canon},
has the extinction property at 
$k = \frac{4 \pi}{a+b} \frac{t}{p} \in \mathbb{S}$ 
if 
$$\lim_{i \to +\infty} \, \frac{\nu_{\mathcal{V},i}(k) dk}{\mbox{card}(
V_{i}
)^{\alpha} dk}~=~ 0$$
for the exponent $\alpha$ which corresponds 
uniquely to $k$ in the (AGL)-argument.

\begin{proposition}
\label{soussuiteannul}
Let $k = \frac{4 \pi}{a+b} \frac{t}{p} \in \mathbb{S}$ with 
$p$ odd. 
If the origin $0$ is not contained in the 
rarefaction domain R$(k)$, then there exists no sequence
of finite approximants $\mathcal{V}$ in
$\mathcal{A}_{a,b}^{supp(\omega)}$ which possesses the extinction property
at $k$.
\end{proposition}

\begin{proof}
Assume that such a sequence
of finite approximants $\mathcal{V} = (V_i)_i$ exists. Then
we would have
\begin{equation}
\label{mieux1exti}
\liminf_{i \to +\infty} \frac{\nu_{\mathcal{V},i}(k) dk}{\mbox{card}(
V_{i}
)^{\alpha} dk}
~=~
\lim_{i \to +\infty} \, \frac{\nu_{\mathcal{V},i}(k) dk}{\mbox{card}(
V_{i}
)^{\alpha} dk}~=~ 0
~=~
\left|\kappa_{\eta}\Bigl(\frac{4 \pi}{a+b} \frac{t}{p}\Bigr)\right|^2
\times \inf_{z \in \mbox{{\scriptsize R}}(k)} \, |z|^2 \, . 
\end{equation}
But $\inf_{z \in \mbox{{\scriptsize R}}(k)} \, |z|^2 \neq 0$ and
$\kappa_{\eta}\Bigl(\frac{4 \pi}{a+b} \frac{t}{p}\Bigr) \neq 0$.
Contradiction. 
\end{proof}

Now let $k = \frac{4 \pi}{a+b} \frac{t}{p} \in \mathbb{S}$
and assume $0 \in ~$R$(k)$.
The problem of finding a sequence of finite approximants 
$\mathcal{V} = (V _i)$, with
$V_i = [0, f(N_i \, p)] \cap (\Lambda_{a,b} \setminus \{0\})$,
which has the extinction property  
at $k$ amounts to
exhibit explicitly an increasing sequence $(N_i)$ of positive 
integers which satisfies 
$$\lim_{i \to \infty} \,  \sum_{j=0}^{p-1} \, \psi_{p,j}\left(
\frac{\log (N_i \, p)}{r s \log 2}\right) \, \xi_{k}^{j} ~=~ 0 ,  
$$
by \eqref{gkspeer}, and in particular sequences $(N_i)$
for which
$$\sum_{j=0}^{p-1} \, \psi_{p,j}\left(
\frac{\log (N_i \, p)}{r s \log 2}\right) \, \xi_{k}^{j} ~=~ 0\qquad \mbox{for all}~ i$$
as soon as $i \geq i_0$ for a certain $i_0$.
This problem concerns at the same time all the
$p$ functions $\psi_{p,j}(x)$ (whose individual cancellation properties 
are recalled in Section \ref{S4}) and seems as difficult as finding
particular sequences of integers 
which realize the min and the max of the
fractal functions $\psi_{p,j}(x)$
(see Coquet \cite{coquet} for $p = 3$ and $j=0$).

Let us notice that most of the sequences of finite approximants
have not the extinction property at $k$.

\subsection{Growth regimes of approximant measures and
visibility in the spectrum}  
\label{S5.3}

\begin{definition}
\label{3regimes}
Let  $k = \frac{4 \pi}{a+b} \frac{t}{2^h \, p} \in \mathbb{S}$. 
We say that the sequence of approximant measures 
$\Bigl(\nu_{\mathcal{U},l}(k) dk\Bigr)_{l \geq 0}$
is
\begin{itemize}
\item[(i)]  {\it size-increasing at $k$} : if and only if $\alpha \in (0, 1)$,    
\item[(ii)] {\it \'etale at $k$} : if and only if $\alpha = 0$,  
\item[(iii)] {\it size-decreasing at $k$} : if and only if $\alpha \in (-1,0)$. 
\end{itemize}
\end{definition}  

These three regimes correspond exactly, by \eqref{alphamieux} and 
\eqref{alphamieuxii}, 
to the three cases of exponents
$$1/2 < \beta < 1, ~\beta = 1/2, ~0 < \beta < 1/2$$  
relative to the asymptotic laws of 
the $p$-rarefied sums of the Thue-Morse sequence.

On the other hand, 
it is usual in physics \cite{cowley} \cite{guinier} \cite{luck}
to represent the function 
``intensity per diffracting site"
$I_{l}(k)$ 
(see \eqref{intensityI}) as a function of $k$, for various 
values of sizes $l = ~$card$(U_l)$ of finite approximant point sets of
$\mu$. In this respect, it is of 
common use to increase $l$ in order to
have a more precise understanding of the spectrum. Here,
it is illusory to do so since, 
on the set of $k$s
at which the sequence of approximant
measures is ``\'etale" and ``size-decreasing", 
the functions
$$l ~~\longrightarrow~~ I_{l}(k)$$
tend to zero much more rapidly than 
on the set of $k$s
at which the sequence of approximant
measures is ``size-increasing": it suffices to compare
\eqref{intensityI} and \eqref{approxlnu} to observe this.
This leads to a
unique visibility in the spectrum, when $l$ is large enough, 
of the singular peaks 
at the $k$s for which the sequence of approximant measures
is ``size-increasing", and  
of the Bragg peaks
characterized by the fact that limsup$_{l \to +\infty} I_{l}(k)$
is a nonzero constant, 
by the Bombieri-Taylor Conjecture.

In the numerical study \cite{peyrierecockayneaxel} 
most of the peaks in the spectrum, 
which are not
Bragg peaks, are guessed to be labeled   
using ``$2^h p$" at the denominator of the wave vector 
with the prime number $p = 3$, not with
any other odd prime.  
Therefore it is important to characterize 
the set of odd prime numbers $p$ 
for which 
the $p$-rarefied sums of the Thue-Morse sequence have the property that
the sequences of approximant measures are ``size-increasing" at
$k = \frac{4 \pi}{a+b} \frac{t}{2^h p} \in \mathbb{S}$.  
In Subsection \ref{S5.4}
we ask this general question and solve it 
for some classes of prime numbers, 
using \cite{goldsteinkellyspeer}. 

\subsection{Classes of prime numbers}
\label{S5.4}

Denote by $\mathbb{P} = \{3, 5, \ldots\}$ 
the set of odd prime numbers.  
In the following, we only consider
prime numbers 
$p$ which are odd, so that $2$ is invertible modulo $p$. 
The order of 
$2$ in the multiplicative group
$\bigl(\mathbb{Z}/ p \mathbb{Z}\bigr)^{*}$, 
denoted
by $s$, is $\geq 1$.

\vspace{0.3cm}

{\bf Problem} : 
Characterize the set of odd prime numbers $p$ for which 
the exponent $\beta = (\log \lambda_1)/(s \log 2)$  
is $> 1/2$ (recall that $\lambda_1$ is the first 
magnitude of eigenvalues of the matrix
$M = (S_{p, i-j}(2^s))_{0 \leq i, j \leq p-1}$).     

\subsubsection{The class $\mathcal{P}_1$} 
\label{S5.4.1}

Let $\mathcal{P}_{1} := \{ p \in \mathbb{P} \mid
s = p-1\}$. 
We have 
$$\mathcal{P}_{1} = \{   
3, 5, 11, 13, 19, 29, 37, 53, 59, 61, 67, 83, 101, 107, 131, 139, 
$$
$$\mbox{} \hspace{5cm} 149, 163, 173, 179, 181, 197,\ldots\}.$$
It is a conjecture of Artin \cite{lenstra}
that
$\mathcal{P}_{1}$ is infinite.
The moduli of the eigenvalues of the matrix 
$M = (S_{p,i-j}(2^{p-1}))_{0 \leq i,j \leq p-1}$ are
$$\lambda_1 ~=~ p~( \mbox{which is}~ (p-1)\mbox{-fold degenerated}),\quad \quad \lambda_2 ~=~ 0.$$
The integer $r$ is equal to $1$, the error terms $E_{p,j}(n)$
are bounded, and 
\eqref{gkspeer} reads
as
\begin{equation}
\label{gkspeerP1}
S_{p,j}(n) ~=~ n^{\beta} \psi_{p,j}\left(\frac{\log n}{ (p-1) \, \log 2}\right)
- \frac{1}{p} \bigl(
\frac{1 - (-1)^n}{2}\bigr) \eta_n,
\end{equation}
with $$\beta ~=~ \frac{\log p}{(p-1) \log 2}$$
with all functions $\psi_{p,j}(x)$
continuous, nowhere differentiable, of period $1$.

\begin{proposition}
\label{pp1}
The prime numbers $p \in \mathcal{P}_{1}$ for which $\beta > 1/2$ 
are $3$ and $5$.
\end{proposition}

\begin{proof}
Indeed,
these values of $p$ are given by the inequality:
\begin{equation}
\label{premierP1incr}
\frac{\log p}{(p-1) \log 2} > \frac{1}{2}\, .
\end{equation}
Since  $x \to \frac{\log x}{(x-1) \log 2}$
goes to zero when $x$ tends to infinity,
only a finite number of values of $p$ satisfy
\eqref{premierP1incr}.
An easy computation gives $3$ and $5$.
\end{proof}

\subsubsection{The class $\mathcal{P}_{2,1}$}
\label{S5.4.2}

Let $\mathcal{P}_{2,1} := \{ p \in \mathbb{P} \mid
s = \frac{p-1}{2}, ~p \equiv 1 \, (\mbox{mod}~ 4) \}$.
We have
$$\mathcal{P}_{2, 1} = \{
17, 41, 97, 137, 193, \ldots \}.$$
The moduli of the eigenvalues of the matrix
$M = (S_{p,i-j}(2^{\frac{p-1}{2}}))_{0 \leq i,j \leq p-1}$ are
$$\lambda_1 ~=~ \epsilon^{h} \sqrt{p} \quad( \mbox{with degeneracy} \geq \frac{p-1}{2}),\quad \quad 
\lambda_2 ~=~ \epsilon^{-h} \sqrt{p} ~\in (0, 1) ,$$
where $h$ is equal to the class number of the field
$\mathbb{Q}(\sqrt{p})$ and $\epsilon > 1$ the fundamental unit in
the real quadratic field $\mathbb{Q}(\sqrt{p})$.
The integer $r$ is equal to $1$, the error terms $E_{p,j}(n)$
are bounded in modulus by
$$2^{\frac{p-1}{2}} \frac{1}{\sqrt{p} - \epsilon^{h}},$$ 
and
\eqref{gkspeer} reads
as
\begin{equation}
\label{gkspeerP1}
S_{p,j}(n) ~=~ n^{\beta} \psi_{p,j}\left(\frac{2 \log n}{ (p-1) \, \log 2}\right)
+ 
E_{p,j}(n)
\end{equation}
with $$\beta ~=~ \frac{\log p + 2 h \log \epsilon}{(p-1) \log 2}$$
with all functions $\psi_{p,j}(x)$
continuous, nowhere differentiable, of period $1$.

\begin{proposition}
\label{pp21}
Only finitely many prime numbers $p$ in $\mathcal{P}_{2,1}$ give rise to 
the inequality $\beta > 1/2$. There is only one such prime number:  
$p = 17$. 
\end{proposition}

\begin{proof}
Let $L(z,\chi_p)$ the Dirichlet $L$-function with
character $\chi_p$ (\cite{borevitchchafarevitch} p. 380). 
The discriminant of $\mathbb{Q}(\sqrt{p})$ is equal to $p$ and
we have the analytic class number formula
(\cite{borevitchchafarevitch} p. 385)
\begin{equation}
\label{clfor}
2 h \log \epsilon ~=~ \sqrt{p} \, L( 1, \chi_p ).
\end{equation}
By Hua's inequality (\cite{hua} Theorem 13.3, p. 328)
$$L(1, \chi_p ) ~<~ \frac{\log p}{2} + 1$$
we deduce
$$\beta ~<~ 
\frac{\log p + \sqrt{p} \, \bigl(\frac{\log p}{2} + 1 \bigr)}{(p-1) \log 2}.$$ 
Since the function
\begin{equation}
\label{foncpp21}
x \to \frac{\log x + \sqrt{x} \, \bigl(\frac{\log x}{2} + 1 \bigr)}{(x-1) \log 2}
\end{equation}
tends to zero when $x$ goes to infinity, we deduce the claim.

To deduce the possible values we notice that
\eqref{foncpp21} intersects the line $y = 1/2$ between $x = 125$ and $126$.
The Table below, obtained from \cite{borevitchchafarevitch} pp. 472--474, and
PARI-GP, gives the first couples $(p, \beta)$ (with $\omega = (1 + \sqrt{p})/2$). 

\begin{center}
\begin{tabular}{c|ccc|cc}
\hline
$p$ & $17$ & $41$ & $97$ & $137$ & \mbox{} \quad $197\quad \ldots$ \\ \hline
$h$ & $1$ & $1$ & $1$ & $1$ & $1$ \\
$\epsilon$ & $3 + 2 \omega$ & $27 + 10 \omega$ & $5035 + 1138 \omega$ & $1595 + 298 \omega$ & $1637147 + 253970 \omega$\\
$\beta$ & 0.6332.. & 0.4339.. & 0.3490.. & 0.2398.. & 0.2672..\\
\end{tabular}
\end{center}
\end{proof}

A complete investigation of the values of $\beta$, 
for all $p \in \mathcal{P}_{2,1}$,
requires the knowledge of the class number $h$
and the regulator $\log \epsilon$ in real quadratic fields, 
in
the Cohen-Lenstra heuristics
\cite{cohenlenstra1} \cite{cohenlenstra2}
\cite{cohenmartinet}.

\subsubsection{The class $\mathcal{P}_{2,3}$}
\label{S5.4.3} 

Let $\mathcal{P}_{2,3} := \{ p \in \mathbb{P} \mid
s = \frac{p-1}{2}, ~p \equiv 3 \, (\mbox{mod}~ 4) \}$.
We have
$$\mathcal{P}_{2,3} = \{
7, 23, 47, 71, 79, 103, 167, 191, 199, \ldots \}.$$
All the eigenvalues of the matrix
$M = (S_{p,i-j}(2^{\frac{p-1}{2}}))_{0 \leq i,j \leq p-1}$ have the same modulus which is
equal to
$$\lambda_1 ~=~ \sqrt{p}.$$ 
The integer $r$ is equal to $4$, the error terms $E_{p,j}(n)$
are bounded, 
and
\eqref{gkspeer} reads
as
\begin{equation}
\label{gkspeerP1}
S_{p,j}(n) ~=~ n^{\beta} \psi_{p,j}\left(\frac{\log n}{ 2 (p-1) \, \log 2}\right)
-
\frac{1}{p} \bigl(
\frac{1 - (-1)^n}{2}\bigr) \eta_n 
\end{equation}
with $$\beta ~=~ \frac{\log p}{(p-1) \log 2},$$
with all functions $\psi_{p,j}(x)$
continuous, nowhere differentiable, of period $1$.

\begin{proposition}
\label{pp23}
There exists no prime number $p$ in $\mathcal{P}_{2,3}$ 
for which $\beta > 1/2$. 
\end{proposition}

\begin{proof}
Same proof as \eqref{pp1}.
\end{proof}

\subsubsection{An inequality}
\label{S5.4.4} 

For the other classes of prime numbers $p$ than
$\mathcal{P}_{1}$,  
$\mathcal{P}_{2,1}$
or $\mathcal{P}_{2,3}$,  
the corresponding exponents $\beta$ which control the rarefaction
phenomenon have the following property (\cite{goldsteinkellyspeer} p. 14):
$$\beta ~=~ \frac{\log \lambda_1}{s \log 2} 
> \frac{\log p}{(p-1) \log 2}.$$
It is expected that 
the number of such odd prime numbers for which $\beta > 1/2$
is zero or eventually is finite.

\section{Other Dirac combs and Marcinkiewicz classes}
\label{S6}

In the case of the equally weighted sum
$\mu = \sum_{n \geq 1} \, \delta_{f(n)}$ 
the fractality of the sum-of-digits functions $\psi_{p,j}(x)$ 
used for describing the singular 
continuous part of the spectrum of $\mu$
arises from the approximation, given by
\eqref{equaapproxF},  of the solution of the
matrix equation \eqref{equajuste}. 
In the general case of a weighted Dirac comb on
the Thue-Morse quasicrystal 
$\Lambda_{a,b}$, the computation of scaling exponents arises from the
existence of
a matrix equation like \eqref{equajuste} which correlates the
block structure of the Prouhet-Thue-Morse sequence by the powers of $2$
and the 
sequence of the weights $(\omega(n))_{n \in \mathbb{Z}}$; then
the computation of solutions 
by interpolation as in \eqref{gkspeer}, with fractal powers of $n$,
 becomes possible.
The Bragg and singular continuous components 
can be deduced in a similar way.

Now let us turn to classes of weighted Dirac combs on $\Lambda_{a,b}$.
Let $\mathcal{L}$ be the space of 
complex-valued functions
on $\Lambda_{a,b}$ 
(or equivalently on $\Z$ through $n \rightarrow f(n)$). 
For 
$w \in \mathcal{L}$, we denote by
$\|w\|$ the pseudo-norm (``norm 1")
of Marcinkiewicz of $w$ defined as
$$\|w\| ~=~ 
\limsup_{l \to +\infty}
~\frac{1}{\mbox{Card}(U_l)}\,
\sum_{n \in \zb, f(n) \in U_l} |w(n)|
$$
where $(U_l)_l$ is an averaging 
sequence of finite approximants
of $\Lambda_{a,b}$.
The Marcinkiewicz space $\mc$
is the quotient space of the subspace
$$\{g \in \mathcal{L} \mid \|g\|
< +\infty \}$$
of $\mathcal{L}$ by the 
equivalence relation $\rc$ defined by
\begin{equation}
\label{marcinkrela}
h ~\rc~ g ~\Longleftrightarrow~ \|h - g\| = 0
\end{equation}
(Bertrandias \cite{bertrandias1}
\cite{bertrandias2}, 
Vo Khac \cite{vokhac}).
The class of $w$ is denoted by 
$\overline{w}$ in $\mc$.
Though the definition of 
$\|\cdot\|$ depends upon the 
chosen averaging sequence
$(U_l)_l$ the space $\mc$ 
obviously does not.
This equivalence relation is called
Marcinkiewicz equivalence relation.
The vector space $\mc$ is normed
with $\|\overline{g}\| =
\|g\|$, and is complete
(Bertrandias \cite{bertrandias1}
\cite{bertrandias2}, 
Vo Khac \cite{vokhac}).
By $\mathcal{L}^{\infty}$ we will mean
the subspace of $\mathcal{L}$ of
bounded weights endowed with the
$\mc$-topology in the sequel.

The following proposition shows 
that the Bombieri-Taylor (BT) argument is compatible
with the Marcinkiewicz equivalence relation $\rc$, i.e.  
that two weighted Dirac combs on $\Lambda_{a,b}$ which are
Marcinkiewicz-equivalent present the same
Bragg component in the spectrum.

\begin{proposition}
\label{intnorme}
Let $h, g \in \mathcal{L}$ such that 
$\overline{h}, \overline{g} \in \mc$. Then, for
$q \in \rb$,
\begin{equation}
\label{inten1}
I_{h}(q) ~\leq~ \|h\|^2, 
\end{equation}
\begin{equation}
\label{inten2}
h ~\rc~ g ~~\Longrightarrow~~ I_{h}(q) ~=~ I_{g}(q).
\end{equation}
\end{proposition}

\begin{proof}
Immediate.
\end{proof}

The relation \eqref{inten2} means that
the set of weighted Dirac combs on the Thue-Morse quasicrystal 
is classified by the Marcinkiewicz relation.
The intensity function $I_w$ (per diffracting site)
is a class function
on $\mc$, when running over the Bragg component of the spectrum.

\frenchspacing

\section*{Acknowledgements}

The authors are indebted to Pierre Liardet 
for  valuable comments, and to the developpers of the PARI/GP system
on which numerical experiments were performed.

\vspace{1cm}
\noindent
Jean Pierre GAZEAU\\
Universit\'e Paris 7 Denis-Diderot,\\
APC - UMR CNRS 7164\\
Boite 7020\\
75251 Paris Cedex 05,
France\\
email:\,{\tt gazeau@apc.univ-paris7.fr}

\vspace{0.2cm}

\noindent
Jean-Louis VERGER-GAUGRY\\
Universit\'e de Grenoble I,\\
Institut Fourier, CNRS UMR 5582, \\
BP 74, Domaine Universitaire,\\
38402 Saint-Martin d'H\`eres, France\\
email:\,{\tt jlverger@ujf-grenoble.fr}


\frenchspacing

\begin{thebibliography}{99}

\bibitem[AMF]{allouchemendesfrance}
    \textsc{J.-P. Allouche {\rm and} M. Mend\`es-France},
    {\it Automata and automatic sequences}, in {\it Beyond Quasicrystals},
    Ed. F. Axel and D. Gratias, Course 11, 
    Les Editions de Physique, Springer (1995), 293--367. 

\bibitem[AGL]{aubrygodrecheluck}
    \textsc{S. Aubry, C. Godr\`eche {\rm and} J.-M. Luck},
    {\it Scaling Properties of a Structure 
    Intermediate between Quasiperiodic and Random},
    J. Stat. Phys. {\bf 51} (1988), 1033--1075.

\bibitem[AT]{axelterauchi}
    \textsc{F. Axel {\rm and} H. Terauchi},
    {\it High-resolution X-ray-diffraction spectra of 
    Thue-Morse GaAs-AlAs heterostructures: Towards a novel description of disorder},
    Phys. Rev. Lett. {\bf 66} (1991), 2223--2226.

\bibitem[B]{bai}
    \textsc{Zai-Qiao Bai},
    {\it Multifractal analysis of the spectral measure of the Thue-Morse
    sequence: a periodic orbit approach},
    J. Phys. A: Math. Gen. {\bf 39} (2006) 10959--10973.

\bibitem[Bs1]{bertrandias1}
    \textsc{J.-P. Bertrandias},
    {\it Espaces de fonctions continues et born\'ees 
    en moyenne asymptotique d'ordre $p$},
    M\'emoire Soc. Math. france (1966), no. 5, 3--106.

\bibitem[B-VK]{bertrandias2}
    \textsc{J.-P. Bertrandias, J. Couot, J. Dhombres, 
    M. Mend\`es-France, P. Phu Hien {\rm and} K. Vo Khac},
    {\it Espaces de Marcinkiewicz, corr\'elations, mesures, syst\`emes dynamiques},
    Masson, Paris (1987). 

\bibitem[BT1]{bombieritaylor1}
    \textsc{E. Bombieri {\rm and} J.E. Taylor},
    {\it ``Which distributions of matter diffract ? An initial
    investigation"},
    J. Phys. Colloque {\bf 47} (1986), C3, 19--28. 

\bibitem[BT2]{bombieritaylor2}
    \textsc{E. Bombieri {\rm and} J.E. Taylor},
    {\it Quasicrystals, tilings, and algebraic 
    number theory: some preliminary connections},
    Contemp. Math. {\bf 64} (1987), 241--264. 

\bibitem[BS]{borevitchchafarevitch}
    \textsc{Z.I. Borevitch {\rm and} I.R. Chafarevitch},
    {\it Th\'eorie des Nombres},
    Gauthiers-Villars, Paris (1967). 

\bibitem[CSM]{chengsavitmerlin}
    \textsc{Z. Cheng, R. Savit {\rm and} R. Merlin},
    {\it Structure and electronic properties of Thue-Morse lattices},
    Phys. Rev B {\bf 37} (1988), 4375--4382.

\bibitem[CL1]{cohenlenstra1}
    \textsc{H. Cohen {\rm and} H.W. Lenstra, Jr.},
    {\it Heuristics on Class groups},
    Lect. Notes Math. {\bf 1052} (1984), 26--36. 

\bibitem[CL2]{cohenlenstra2}
    \textsc{H. Cohen {\rm and} H.W. Lenstra, Jr.},
    {\it Heuristics on Class groups of number fields},
    Number Theory, Proc. Journ. Arithm., Noodwijkerhout/Neth. 1983,
    Lect. Notes Math. {\bf 1068} (1984), 33--62.

\bibitem[CM]{cohenmartinet}
    \textsc{H. Cohen {\rm and} J. Martinet},
    {\it Class Groups of Number Fields: Numerical Heuristics},
    Math. Comp. {\bf 48} (1987), 123--137. 

\bibitem[Ct]{coquet}
    \textsc{J. Coquet},
    {\it A summation formula related to the binary digits},
    Inv. Math. {\bf 73} (1983), 107--115.  

\bibitem[Cy]{cowley}
    \textsc{J.-M. Cowley},
    {\it Diffraction physics},
    North-Holland, Amsterdam (1986), 2nd edition.

\bibitem[DS1]{drmotaskalba1}
    \textsc{M. Drmota {\rm and} M. Skalba},
    {\it Sign-changes of the Thue-Morse fractal fonction
    and Dirichlet $L$-series},
    Manuscripta Math. {\bf 86} (1995), 519--541.

\bibitem[DS2]{drmotaskalba2}
    \textsc{M. Drmota {\rm and} M. Skalba},
    {\it Rarefied sums of the Thue-Morse sequence},
    Trans. Amer. Math. Soc. {\bf 352} (2000), 609--642.   

\bibitem[D]{dumont}
    \textsc{J.M. Dumont},
    {\it Discr\'epance des progressions arithm\'etiques 
    dans la suite de Morse},
    C. R. Acad. Sci. Paris 
    S\'erie I {\bf 297} (1983), 145--148. 

\bibitem[GVG]{gazeauvergergaugry}
    \textsc{J. P. Gazeau {\rm and} J.-L. Verger-Gaugry},
    {\it Diffraction spectra of weighted Delone sets on beta-lattices
    with beta a quadratic unitary Pisot number},
    Ann. Inst. Fourier {\bf 56} (2006), 2437--2461.

\bibitem[Gd]{gelfond}
    \textsc{A.O. Gelfond},
    {\it Sur les nombres qui ont des propri\'et\'es additives et multiplicatives donn\'ees},
    Acta Arith. {\bf 13} (1968), 259--265.

\bibitem[GL1]{godrecheluck1}
    \textsc{C. Godr\`eche {\rm and} J.-M. Luck},
    {\it Multifractal analysis in reciprocal space and the nature of the Fourier
    transform of self-similar structures},
    J. Phys. A: Math. Gen. {\bf 23} (1990), 3769--3797. 

\bibitem[GL2]{godrecheluck2}
    \textsc{C. Godr\`eche {\rm and} J.-M. Luck},
    {\it Indexing the diffraction spectrum of a non-Pisot self-structure},
    Phys. Rev. B {\bf 45} (1992), 176--185. 

\bibitem[GKS]{goldsteinkellyspeer}
    \textsc{S. Goldstein, K.A. Kelly {\rm and} E.R. Speer},
    {\it The fractal structure of rarefied sums of the Thue-Morse sequence},
    J. Number Th. {\bf 42} (1992), 1--19.

\bibitem[Gr1]{grabner1}
    \textsc{P.J. Grabner},
    {\it A note on the parity of the sum-of-digits function},
    Actes 30i\`eme S\'eminaire Lotharingien de Combinatoire (Gerolfingen, 1993), 35--42.

\bibitem[Gr2]{grabner2}
    \textsc{P.J. Grabner},
    {\it Completely $q$-Multiplicative Functions: the Mellin Transform Approach},
    Acta Arith. {\bf 65} (1993), 85--96. 

\bibitem[GHT]{grabnerherenditichy}
    \textsc{P.J. Grabner, T. Herendi {\rm and} R.F. Tichy},
    {\it Fractal digital sums and Codes},
    AAECC {\bf 8} (1997), 33-39.

\bibitem[G]{guinier}
    \textsc{A. Guinier},
    {\it Theory and Technics for X-Ray Crystallography},
    Dunod, Paris (1964).

\bibitem[H]{hof}
    \textsc{Hof},
    {\it On diffraction by aperiodic structures},
    Commun. Math. Phys. {\bf 169} (1995), 25--43.

\bibitem[H]{hua}
    \textsc{L.-K. Hua},
    {\it Introduction to Number Theory},
    Springer-Verlag, Berlin-New York (1982).

\bibitem[K]{kac}
    \textsc{M. Kac},
    {\it On the distribution of values of sums of the type
    $\sum f(2^k t)$},
    Ann. Math. {\bf 47} (1946), 33--49.

\bibitem[KIR]{kolariochumraymond}
    \textsc{M. Kol\'ar, B. Iochum {\rm and} L. Raymond},
    {\it Structure factor of 1D systems (superlattices) based on two-letter
    substitution rules: I. $\delta$ (Bragg) peaks},
    J. Phys. A: Math. Gen. {\bf 26} (1993), 7343--7366. 

\bibitem[La1]{lagarias1}
    \textsc{J.C. Lagarias},
    {\it Meyer's concept of quasicrystal and quasiregular sets},
    Comm. Math. Phys. {\bf 179} (1995), 365--376.  

\bibitem[La2]{lagarias2}
    \textsc{J.C. Lagarias},
    {\it Mathematical quasicrystals and the problem of diffraction},
    in {\it Directions in Mathematical Quasicrystals},
    ed. M. Baake \& R.V. Moody, CRM Monograph Series, Amer. Math. Soc.
    Providence, RI, (2000), 61--93.  

\bibitem[Le]{lenstra}
    \textsc{H.W. Lenstra Jr.}, 
    {\it On Artin's conjecture and Euclid's algorithm in global fields},  
    Invent. Math. {\bf 42}  (1977), 201--224. 

\bibitem[Lz]{lenz}
    \textsc{D. Lenz},
    {\it Continuity of Eigenfonctions of Uniquely Ergodic Dynamical Systems
    and Intensity of Bragg peaks},
    preprint (2006). 

\bibitem[Lu]{luck}
    \textsc{J.-M. Luck},
    {\it Cantor spectra and scaling of gap 
    widths in deterministic aperiodic systems},
    Phys. Rev. B {\bf 39} (1989), 5834--5849.

\bibitem[M]{moody}
    \textsc{R.V. Moody},
    {\it Meyer sets and their duals}, in
    {\it The Mathematics of Long-Range Aperiodic Order},
    Ed. R.V. Moody, Kluwer (1997), 403--442. 

\bibitem[N]{newman}
    \textsc{D.J. Newman},
    {\it On the number of binary digits in a multiple of three},
    Proc. Am. Math. Soc. {\bf 21} (1969), 719--721.  



\bibitem[Oa]{deoliveira}               
    \textsc{C.R. de Oliveira},
    {\it A proof of the dynamical version of the Bombieri-Taylor Conjecture},
    J. Math. Phys. {\bf 39} (1998), 4335--4342.

\bibitem[P]{peyriere}
    \textsc{J. Peyri\`ere},
    {\it Etude de quelques propri\'et\'es des produits de Riesz},
    Ann. Inst. Fourier {\bf 25} (1975), 127--169.

\bibitem[PCA]{peyrierecockayneaxel}
    \textsc{J. Peyri\`ere, E. Cockayne {\rm and} F. Axel},
    {\it Line-Shape Analysis of High Resolution X-Ray Diffraction
    Spectra of Finite Size Thue-Morse GaAs-AlAs Multilayer Heterostructures},
    J. Phys. I France {\bf 5} (1995), 111--127. 

\bibitem[Q1]{queffelec1}
    \textsc{M. Queff\'elec},
    {\it Dynamical systems - Spectral Analysis},
    Lect. Notes Math. {\bf 1294} (1987). 

\bibitem[Q2]{queffelec2}
    \textsc{M. Queff\'elec},
    {\it Spectral study of automatic and substitutive sequences},
    in {\it Beyond Quasicrystals}, 
    Ed. F. Axel and D. Gratias, Course 12, 
    Les Editions de Physique - Springer (1995), 369--414.

\bibitem[R]{raikov}
    \textsc{D. Raikov},
    {\it On some arithmetical properties of summable functions},
    Rec. Math. de Moscou {\bf 1} (43;3) (1936) 377--383.

\bibitem[Sz]{schwartz}
    \textsc{L. Schwartz},
    {\it Th\'eorie des distributions},
    Hermann, Paris (1973).

\bibitem[Su]{strungaru}
    \textsc{Strungaru},
    {\it Almost Periodic Measures and 
    Long-Range Order in Meyer Sets},
    Discr. Comput. Geom. {\bf 33} (2005), 483--505.

\bibitem[VG]{vergergaugry}
    \textsc{J.-L. Verger-Gaugry},
    {\it On self-similar finitely generated uniformly 
    discrete (SFU-) sets and sphere packings},  
    in IRMA Lect. in Math. and Theor. Phys. {\bf 10}, Ed. L. Nyssen, 
    E.M.S. (2006), 39--78.

\bibitem[VK]{vokhac}
    \textsc{K. Vo Khac},
    {\it Fonctions et distributions stationnaires. Application \`a 
    l'\'etude des solutions stationnaires 
    d'\'equations aux d\'eriv\'ees partielles},
    in \cite{bertrandias2}, pp 11--57. 


\bibitem[WWVG]{wolnywnekvergergaugry}
    \textsc{J. Wolny, A. Wnek {\rm and} J.-L. Verger-Gaugry},
    {\it Fractal behaviour of diffraction patterns of Thue-Morse sequence},
    J. Comput. Phys. {\bf 163} (2000), 313.
 
\end{thebibliography}
\end{document}